\documentclass[10pt]{amsart}
\usepackage{amsfonts,amscd,amsthm,amsgen,amsmath,amssymb}
\usepackage{epstopdf}
\usepackage{epsfig,color}
\usepackage[all]{xy}
\usepackage[vcentermath]{youngtab}
\usepackage{tikz}
\usetikzlibrary{decorations.markings}
\usepackage[letterpaper]{geometry}
\newtheorem{theorem}{Theorem}[section]
\newtheorem{lemma}[theorem]{Lemma}
\newtheorem{corollary}[theorem]{Corollary}
\newtheorem{proposition}[theorem]{Proposition}

\theoremstyle{definition}
\newtheorem{definition}[theorem]{Definition}

\theoremstyle{remark}
\newtheorem{remark}[theorem]{Remark}

\numberwithin{equation}{section}

\begin{document}
\tikzset{->-/.style={decoration={
  markings,
  mark=at position #1 with {\arrow{>}}},postaction={decorate}}}

\tikzset{-<-/.style={decoration={
  markings,
  mark=at position #1 with {\arrow{<}}},postaction={decorate}}}

\title[Monodromy of the $SL(n)$ and $GL(n)$ Hitchin fibrations]{Monodromy of the $SL(n)$ and $GL(n)$ Hitchin fibrations}
\author{David Baraglia}

\address{School of Mathematical Sciences, The University of Adelaide, Adelaide SA 5005, Australia}
\email{david.baraglia@adelaide.edu.au}

\begin{abstract}
We compute the monodromy of the Hitchin fibration for the moduli space of $L$-twisted $SL(n,\mathbb{C})$ and $GL(n,\mathbb{C})$-Higgs bundles for any $n$, on a compact Riemann surface of genus $g>1$. We require the line bundle $L$ to either be the canonical bundle or satisfy $deg(L) > 2g-2$. The monodromy group is generated by Picard-Lefschetz transformations associated to vanishing cycles of singular spectral curves. We construct such vanishing cycles explicitly and use this to show that the $SL(n,\mathbb{C})$ monodromy group is a {\em skew-symmetric vanishing lattice} in the sense of Janssen. Using the classification of vanishing lattices over $\mathbb{Z}$, we completely determine the structure of the monodromy groups of the $SL(n,\mathbb{C})$ and $GL(n,\mathbb{C})$ Hitchin fibrations. As an application we determine the image of the restriction map from the cohomology of the moduli space of Higgs bundles to the cohomology of a non-singular fibre of the Hitchin fibration.
\end{abstract}





\maketitle


\section{Introduction}

\subsection{Monodromy of the Hitchin fibration}\label{secmonohit}
In this paper we determine the monodromy of the $SL(n,\mathbb{C})$-Hitchin fibration $h : \mathcal{M}(n,L) \to \mathcal{A}$ for all $n$. Here $\mathcal{M}(n,L)$ is the moduli space of $L$-twisted $SL(n,\mathbb{C})$-Higgs bundles on a compact Riemann surface $\Sigma$ of genus $g > 1$ and $L$ is a line bundle which is either the canonical bundle $K$, or satisfies $l = deg(L) > 2g-2$. Our proof also gives the monodromy of the corresponding $GL(n,\mathbb{C})$-Hitchin fibration. For $SL(2,\mathbb{C})$, the monodromy was first determined in \cite{cop} using a combinatorial approach and later revisited in \cite{bs1} from a more geometric point of view. It does not seem possible to extend the arguments used in \cite{cop}, \cite{bs1} to rank $n > 2$. In this paper we introduce new techniques that apply for all $n$. When $n > 2$, our results are completely new.\\

Recall that the Hitchin fibration is a proper, surjective holomorphic map $h : \mathcal{M}(n,L) \to \mathcal{A}$ from $\mathcal{M}(n,L)$ to the affine space $\mathcal{A} = \bigoplus_{j=2}^n H^0(\Sigma , L^j)$ \cite{hit1,hit2,nit}. As shown by Hitchin \cite{hit2}, when $L = K$ the Hitchin map gives $\mathcal{M}(n,L)$ the structure of an algebraically completely integrable system with respect to a natural holomorphic symplectic structure on $\mathcal{M}(n,L)$. For $L \neq K$, we do not have a holomorphic symplectic structure, but it remains the case that the non-singular fibres of the Hitchin map are abelian varieties. Finding the monodromy of this fibration is a natural problem which has a number of important applications, some of which are described in \textsection \ref{secapp}.\\

To explain our results we will mainly focus on the $SL(n,\mathbb{C})$ case. The $GL(n,\mathbb{C})$ case is covered by Theorem \ref{thmglmon}. Let $\mathcal{D} \subset \mathcal{A}$ denote the locus of singular fibres of the $SL(n,\mathbb{C})$-Hitchin fibration and let $\mathcal{A}^{\rm reg} = \mathcal{A} \setminus \mathcal{D}$ be the regular locus. Let $\mathcal{M}^{\rm reg}$ be the points of $\mathcal{M}$ lying over $\mathcal{A}^{\rm reg}$, so that $h : \mathcal{M}^{\rm reg} \to \mathcal{A}^{\rm reg}$ is a non-singular bundle of abelian varieties. The monodromy of the $SL(n,\mathbb{C})$-Hitchin system is the local system $\underline{\Lambda}$ on $\mathcal{A}^{\rm reg}$ whose fibre over a point $a \in \mathcal{A}^{\rm reg}$ is the underlying lattice $\underline{\Lambda}_a = H_1( h^{-1}(a) , \mathbb{Z} )$ of the abelian variety $h^{-1}(a)$. Equivalently $\underline{\Lambda}$ is the dual of the Gauss-Manin local system $R^1 h_* \mathbb{Z}$. Choose a basepoint $a_0 \in \mathcal{A}^{\rm reg}$ and let $\Lambda_P$ be the fibre of $\underline{\Lambda}$ over $a_0$ (we use subscript $P$ because $\Lambda_P$ is the lattice of a Prym variety, see \textsection \ref{sec2}). The local system $\underline{\Lambda}$ is then equivalent to a representation $\rho_{SL} : \pi_1( \mathcal{A}^{\rm reg} , a_0) \to Aut( \Lambda_P)$. We call $\rho_{SL}$ the {\em monodromy representation} of the Hitchin fibration. The image $\Gamma_{SL} \subseteq Aut( \Lambda_P )$ of $\rho_{SL}$ will be called the {\em monodromy group} of the $SL(n,\mathbb{C})$-Hitchin fibration. The smooth fibres of the Hitchin fibration are Prym varieties associated to certain branched covers of $\Sigma$ called {\em spectral curves}, recalled in \textsection \ref{sec2}. This construction shows that the smooth fibres are equipped with a natural polarization, which defines a non-degenerate skew-symmetric bilinear pairing $\langle \; , \; \rangle : \Lambda_P \times \Lambda_P \to \mathbb{Z}$ invariant under the monodromy representation. Thus the monodromy of the $SL(n,\mathbb{C})$-Hitchin fibration is given by a triple $( \Lambda_P , \langle \; , \; \rangle , \Gamma_{SL})$, consisting of:
\begin{itemize}
\item[(i)]{a lattice $\Lambda_P$,}
\item[(ii)]{a skew-symmetric $\mathbb{Z}$-valued bilinear form $\langle \; , \; \rangle$ on $\Lambda_P$,}
\item[(iii)]{a subgroup $\Gamma_{SL} \subset Aut(\Lambda_P)$ preserving $\langle \; , \; \rangle$.}
\end{itemize}

\subsection{Vanishing cycles}

The monodromy group is generated by a certain collection of vanishing cycles. To describe these vanishing cycles, we choose a basepoint $a_0 \in \mathcal{A}^{\rm reg}$ of the form $a_0 = (0,0, \dots , 0 , a_n)$, where $a_n \in H^0(\Sigma , L^n)$ has only simple zeros. Let $tot(L)$ denote the total space of $L$ and $\pi : tot(L) \to \Sigma$ the projection. The spectral curve $S$ associated to $a_0$ is given by:
\begin{equation*}
S = \{ \lambda \in tot(L) \; | \; \lambda^n + a_n(\pi(\lambda)) = 0 \} \subset tot(L).
\end{equation*}
Spectral curves of this form carry a Galois action of the cyclic group $\mathbb{Z}_n$ by deck transformations, and will thus be referred to as {\em cyclic spectral curves}. One of the key insights of this paper is that the monodromy of the Hitchin fibration can be computed by studying cyclic spectral curves and small perturbations of them.\\

Given a cyclic spectral curve $S$, we construct vanishing cycles as follows. Let $\pi : S \to \Sigma$ denote the restriction of $\pi$ to $S$. Then $\pi : S \to \Sigma$ is a degree $n$ branched cover of $\Sigma$, branched over the zeros $b_1,b_2, \dots , b_k$ of $a_n$, where $k = nl$. Let $u_1, \dots , u_k \in S$ be the corresponding ramification points in $S$. The spectral curve construction identifies the lattice $\Lambda_P$ with the kernel of the Gysin homomorphism $\pi_* : H^1(S , \mathbb{Z}) \to H^1(\Sigma , \mathbb{Z})$. Suppose that $\gamma : [0,1] \to \Sigma$ is an embedded path in $\Sigma$ joining $b_i$ to $b_j$, where $i \neq j$ and assume that $\gamma$ does not meet any other branch point. Then $\pi^{-1}( \gamma( [0,1] ) ) \subset S$ consists of $n$ paths $\gamma^1 , \gamma^2 , \dots , \gamma^n : [0,1] \to S$ joining $u_i$ to $u_j$. Let $t : S \to S$ be the generator of the cyclic Galois action given by $\lambda \mapsto e^{2\pi i/n}\lambda$ and order the paths so that $t \gamma^i = \gamma^{i+1}$, $i = 1 , \dots , n-1$. Then $(\gamma^1 - \gamma^2)$ is a $1$-cycle in $S$. Let $l_\gamma \in H_1(S , \mathbb{Z})$ be its homology class and $c_\gamma \in H^1(S , \mathbb{Z})$ the Poincar\'e dual class. Similarly the cycles $(\gamma^2 -\gamma^3) , \dots , (\gamma^{n-1} - \gamma^n) , (\gamma^n - \gamma^1)$ correspond to the cohomology classes $tc_\gamma , \dots , t^{n-2}c_\gamma , t^{n-1}c_\gamma \in H^1(S,\mathbb{Z})$. Clearly these cycles are in the kernel of $\pi_* : H^1(S , \mathbb{Z}) \to H^1(\Sigma , \mathbb{Z})$, so they are elements of $\Lambda_P$. As the paths $\gamma^1 , \dots , \gamma^n$ are only determined by $\gamma$ up to cyclic permutation, the cycles $c_\gamma , tc_\gamma , \dots , t^{n-1}c_\gamma$ are likewise only determined by $\gamma$ up to cyclic permutation. We call $c_\gamma , tc_\gamma , \dots , t^{n-1}c_\gamma$ the {\em vanishing cycles associated to $\gamma$}. We then have:
\begin{theorem}
Let $a = t^ic_\gamma \in \Lambda_P$ be a vanishing cycle associated to $\gamma$. Then the $SL(n,\mathbb{C})$ monodromy group $\Gamma_{SL}$ contains the Picard-Lefschetz transformation $T_a : \Lambda_P \to \Lambda_P$ associated to $a$, given by:
\begin{equation}\label{equtransvection}
T_a(x) = x + \langle a , x \rangle a.
\end{equation}
\end{theorem}

\subsection{Vanishing lattices}

To state our main results, we need to recall the notion of a {\em skew-symmetric vanishing lattice} \cite{jan1}. Let $R$ denote the ring $\mathbb{Z}$ or $\mathbb{Z}_2$ and let $V$ be a free $R$-module of rank $\mu$ equipped with a bilinear form $\langle \; , \; \rangle : V \times V \to R$ which is alternating, i.e. $\langle x , x \rangle = 0$ for all $x \in V$. For any $a \in V$ we have an endomorphism $T_a : V \to V$, called a {\em symplectic transvection}, given by Equation (\ref{equtransvection}). Clearly $T_a$ acts as an automorphism of $V$ which preserves $\langle \; , \; \rangle$. Let $j : V \to V^*$ be the map $j(x) = \langle x , \; \rangle$, let $V_0$ be the kernel of $j$ and $V'$ the image of $j$. Let $Sp^{\#}V$ denote the automorphisms of $(V , \langle \; , \; \rangle)$ acting trivially on $V^*/V'$ and note that $T_a$ belongs to $Sp^{\#}V$ for any $a \in V$.

\begin{definition}[\cite{jan1}]
Let $\Delta \subseteq V$ and let $\Gamma_\Delta$ be the subgroup of $Sp^{\#}V$ generated by $\{ T_\alpha \}_{\alpha \in \Delta}$. We say that $(V , \langle \; , \; \rangle , \Delta)$ is a {\em (skew-symmetric)-vanishing lattice} over $R$ if:
\begin{enumerate}
\item[(i)]{$\Delta$ is a $\Gamma_\Delta$-orbit.}
\item[(ii)]{$\Delta$ spans $V$.}
\item[(iii)]{If $\mu >1$, then there exists $\delta_1,\delta_2 \in \Delta$ for which $\langle \delta_1 , \delta_2 \rangle = 1$.}
\end{enumerate}
We also call $\Gamma_{\Delta}$ the {\em monodromy group} of the vanishing lattice.
\end{definition}

\subsection{Main Results}\label{secmainres}

Let $\mathcal{VC} \subset \Lambda_P$ be the collection of all vanishing cycles $c_\gamma , \dots , t^{n-1}c_\gamma$ associated to paths $\gamma$ joining pairs of branch points. Let $\Gamma_{\mathcal{VC}}$ be the subgroup of $Sp^{\#}( \Lambda_P )$ generated by transvections $T_a$, where $a \in \mathcal{VC}$ and set $\Delta_P = \Gamma_{\mathcal{VC}} \cdot \mathcal{VC}$.

\begin{theorem}
We have that $(\Lambda_P , \langle \; , \; \rangle , \Delta_P)$ is a vanishing lattice. Let $\Gamma_{\Delta_P}$ be the subgroup of $Sp^{\#}(\Lambda_P)$ generated by transvections in $\Delta_P$. Then $\Gamma_{\mathcal{VC}} = \Gamma_{\Delta_P} = \Gamma_{SL}$. That is, the monodromy group of the $SL(n,\mathbb{C})$-Hitchin fibration is the monodromy group $\Gamma_{\Delta_P}$ of the vanishing lattice $(\Lambda_P , \langle \; , \; \rangle , \Delta_P)$.
\end{theorem}

Thus to describe the monodromy group $\Gamma_{SL}$ and its action on $\Lambda_P$, it suffices to classify the vanishing lattice $(\Lambda_P , \langle \; , \; \rangle , \Delta_P)$. The classification is as follows (here we use the notation for vanishing lattices introduced in \cite{jan1,jan2}, which is recalled in \textsection \ref{secthevanlat}):

\begin{theorem}
The vanishing lattice $(\Lambda_P , \langle \; , \; \rangle , \Delta_P)$ is isomorphic to:
\begin{enumerate}
\item{$A'(1,1, \dots , 1 , 2 , 2 , \dots , 2; 0)$, if $n=2$,}
\item{$O_a^{\#}(1,1, \dots , 1 , n , n , \dots , n ; 0)$, where $a = (m(m-1)/2)l$, if $n =2m+1$ is odd,}
\item{$O_a^{\#}(1,1, \dots , 1 , n , n , \dots , n ; 0)$, where $a = m(l/2)$, if $n=2m$ and $l$ are even and $n>2$,}
\item{$Sp^{\#}(1,1, \dots , 1 , n , n , \dots , n ; 0)$, if $n$ is even, $l$ is odd and $n > 2$.}
\end{enumerate}
In this classification, the number of $1$'s is $(n-2)(g-1) + n(n-1)l/2 - 1$ and the number of $n$'s is $g$.
\end{theorem}

As mentioned previously, our results also yield the monodromy of the $GL(n,\mathbb{C})$-Hitchin fibration. Fix a basepoint $a_0 \in \mathcal{A}^{\rm reg}$ with spectral curve $S$ and let $\Lambda_S = H^1(S, \mathbb{Z})$. Notice that $\Lambda_P$ is a sublattice of $\Lambda_S$. The monodromy of the $GL(n,\mathbb{C})$-Hitchin fibration is given by a representation $\rho_{GL} : \pi_1( \mathcal{A}^{\rm reg} , a_0) \to Aut( \Lambda_S)$. Let $\Gamma_{GL} \subseteq Aut( \Lambda_S)$ be the image of $\rho_{GL}$. We have:

\begin{theorem}\label{thmglmon}
The monodromy group $\Gamma_{GL}$ is the subgroup of $Aut(\Lambda_S)$ generated by transvections $T_{\alpha} : \Lambda_S \to \Lambda_S$, where $\alpha \in \Delta_P$.
\end{theorem}

In \textsection \ref{secappl} we give an application of our monodromy computations. Let $\mathcal{M}(n,d,L)$ denote the moduli space of rank $n$ $L$-twisted Higgs bundles with trace-free Higgs field and determinant equal to a fixed line bundle of degree $d$. The Hitchin fibration $h : \mathcal{M}(n,d,L) \to \mathcal{A}$ may be defined for any value of $d$. Let $F_a = h^{-1}(a)$ be a non-singular fibre of the Hitchin fibration lying over $a \in \mathcal{A}^{\rm reg}$. Then:
\begin{theorem}\label{thmrestriction0}
Let $\omega \in H^2( F_a , \mathbb{Q} )$ be the cohomology class of the polarization on $F_a$. The image
\begin{equation*}
Im( H^*( \mathcal{M}(n,d,L) , \mathbb{Q} )) \to H^*( F_a , \mathbb{Q} ) )
\end{equation*}
of the restriction map in cohomology is the subspace spanned by $1 , \omega , \omega^2 , \dots , \omega^u$, where $u = dim_{\mathbb{C}}(F_a)$ is the dimension of the fibre.
\end{theorem}
This result generalises to $SL(n,\mathbb{C})$ a result proved by Thaddeus for $SL(2,\mathbb{C})$ in \cite{tha} (see also the appendix of \cite{chm}). When $n$ and $d$ are coprime, the moduli space $\mathcal{M}(n,d,L)$ is smooth and Theorem \ref{thmrestriction0} can be proven by a straightforward generalisation of the argument given in \cite{chm}. The argument breaks down when $n$ and $d$ are not coprime, in which case Theorem \ref{thmrestriction0} is a new result.

\subsection{Applications}\label{secapp}
Knowledge of the monodromy is important for applications of the moduli space of Higgs bundles in which the Hitchin fibration plays a prominent role. We describe some of these applications here.\\

{\bf Cohomology of the moduli space and Ng\^o's support theorem}. Let $\mathcal{M}_{GL}(n,d,L)$ denote the moduli space of $L$-twisted Higgs bundles of rank $n$, degree $d$. Assume that $n$ and $d$ are coprime, in which case $\mathcal{M}_{GL}(n,d,L)$ is non-singular. We have the Hitchin fibration $h : \mathcal{M}_{GL}(n,d,L) \to \mathcal{A}_{GL}$ whose monodromy is isomorphic to $\underline{\Lambda}$ for any value of $d$. Let $\mathcal{A}_{GL}^{\rm ell} \subset \mathcal{A}_{GL}$ be the elliptic locus of $\mathcal{A}_{GL}$, the locus of points for which the corresponding spectral curve is reduced and irreducible. Ng\^o's support theorem \cite{ngo} (see also \cite{cm}) applied to the restriction $h^{\rm ell} : \mathcal{M}_{GL}^{\rm ell}(n,d,L) \to \mathcal{A}_{GL}^{\rm ell}$ of $h$ over $\mathcal{A}_{GL}^{\rm ell}$ implies that the perverse sheaves on $\mathcal{A}^{\rm ell}_{GL}$ occuring in the decomposition theorem for $h^{\rm ell}$ are supported on the whole of $\mathcal{A}_{GL}^{\rm ell}$. The decomposition theorem then reduces to:
\begin{equation*}
Rh^{\rm ell}_* \mathbb{Q}[ dim(\mathcal{M}_{GL}(n,d,L)) ] = \bigoplus_i IC_{\mathcal{A}_{GL}^{\rm ell}}( \wedge^i \underline{\Lambda}^* \otimes_{\mathbb{Z}} \mathbb{Q} )[f-i],
\end{equation*}
where $f$ is the dimension of the fibres of the Hitchin fibration. This raises the possibility of computing the cohomology of $\mathcal{M}_{GL}^{\rm ell}(n,d,L)$ through a knowledge of the local system $\underline{\Lambda}$. In turn, this gives us partial information about the cohomology of the full moduli space $\mathcal{M}_{GL}(n,d,L)$. In fact for $d > 2g-2$, it was shown by Chaudouard and Laumon that there are no new supports when $\mathcal{A}_{GL}^{\rm ell}$ is replaced by $\mathcal{A}_{GL}$ \cite{chla}. A similar result holds in the $SL(n,\mathbb{C})$ case, except there are additional supports related to the endoscopy theory of $SL(n,\mathbb{C})$ \cite{cat}. In a related direction, we note that local monodromy calculations were used extensively in \cite{chm} in the proof of the $P = W$ conjecture for $GL(2,\mathbb{C}), SL(2,\mathbb{C}), PSL(2,\mathbb{C})$. We expect our monodromy calculations to be similarly useful in tackling this conjecture for higher rank groups.\\

{\bf Wall crossing and the hyperk\"ahler metric}. In the work of Gaiotto, Moore and Neitzke on the Kontsevich-Soibelman wall-crossing formula for $\mathcal{N} = 2$ supersymmetric quantum field theories, a relation is established between hyperk\"ahler metrics and counts of BPS states \cite{gmn1}. In particular, this can be applied to the hyperk\"ahler metric on the moduli space of Higgs bundles (in the untwisted case $L = K$). Through a twistorial construction, the hyperk\"ahler metric on the moduli space of Higgs bundles is encoded in a family of complex symplectic forms $\omega(\zeta)$ parametrised by $\zeta \in \mathbb{CP}^1$. In \cite{gmn1}, the authors consider local Darboux coordinates for the symplectic forms $\omega(\zeta)$. The Darboux coordinates $\chi_\gamma$, are locally defined $\mathbb{C}^*$-valued functions on $\mathcal{M}^{\rm reg}$ depending on a choice of point $\zeta \in \mathbb{C}^* \subset \mathbb{CP}^1$ and local section $\gamma$ of the monodromy local system $\underline{\Lambda}$. The functions $\chi_\gamma$ depend multiplicatively on $\gamma$ in the sense that $\chi_{\gamma + \gamma'} = \chi_\gamma \chi_{\gamma'}$ and are Darboux coordinates in the sense that their Poisson brackets are given by the formula:
\begin{equation*}
\{ \chi_\gamma , \chi_{\gamma'} \} = \langle \gamma  , \gamma' \rangle \chi_{\gamma} \chi_{\gamma'}.
\end{equation*}
The locally defined coordinates $\chi_\gamma$ do not patch together globally, instead they satisfy a wall-crossing formula as one crosses a real codimension $1$ subspace of $\mathcal{M}^{\rm reg} \times \mathbb{C}^*$. The wall-crossing formula is given by a holomorphic symplectomorphism, which may be expressed in terms of symplectomorphisms $\mathcal{K}_\gamma$ of the form
\begin{equation*}
\mathcal{K}_\gamma : \chi_{\gamma'} \mapsto \chi_{\gamma'}( 1 - (-1)^{q(\gamma)}\chi_{\gamma})^{\langle \gamma' , \gamma \rangle },
\end{equation*}
where $q : \underline{\Lambda} \to \mathbb{Z}_2 = \{ 0, 1 \}$ is a quadratic function on $\underline{\Lambda}$ (see \textsection \ref{secquadratics}). Interestingly, the existence of a monodromy invariant quadratic function on $\underline{\Lambda}$ is a key ingredient in the classification of $\underline{\Lambda}$ as a vanishing lattice. The proposal of \cite{gmn1} is that the existence of coordinates $\chi_\gamma$ satisfying the wall-crossing formula and some further technical conditions completely determines the hyperk\"ahler metric on the moduli space, through twistor theory. Clearly it is important for this proposal to be able to describe the local system $\underline{\Lambda}$ and its quadratic function $q$.\\

{\bf Connected components of real character varieties}. Let $G$ be a real or complex reductive Lie group. A representation $\theta : \pi_1(\Sigma) \to G$ is called {\em reductive} if the action of $\pi_1(\Sigma)$ on the Lie algebra of $G$ obtained by composing $\theta$ with the adjoint representation is a direct sum of irreducible representations. Let $Hom^{\rm red}(\pi_1(\Sigma) , G)$ be the space of reductive representations given the compact-open topology. The group $G$ acts on $Hom^{\rm red}(\pi_1(\Sigma) , G)$ by conjugation and the quotient
\begin{equation*}
Rep(G) = Hom^{\rm red}(\pi_1(\Sigma) , G)/G
\end{equation*}
is a Hausdorff space \cite{ric}, called the {\em character variety} of representations of $\pi_1(\Sigma)$ in $G$. The non-abelian Hodge correspondence gives a homeomorphism between $Rep(G)$ and the moduli space $\mathcal{M}_G$ of polystable $G$-Higgs bundles, with $L = K$. If $G_\mathbb{C}$ is a complex reductive group and $G_{\mathbb{R}}$ the split real form, then as shown by Schaposnik \cite[Theorem 4.12]{sch0}, the image of the natural map $\mathcal{M}_{G_{\mathbb{R}}} \to \mathcal{M}_{G_{\mathbb{C}}}$ meets the smooth fibres of the Hitchin fibration in points of order $2$. The points of order $2$ in smooth fibres are described by the monodromy with $\mathbb{Z}_2$-coefficients $\underline{\Lambda}[2] = \underline{\Lambda} \otimes_{\mathbb{Z}} \mathbb{Z}_2$. Thus one may hope to determine the number of connected components of $\mathcal{M}_{G_{\mathbb{R}}}$ and hence of $Rep(G_{\mathbb{R}})$ by counting the number of orbits of the local system $\underline{\Lambda}[2]$. This strategy has been successfully carried out for $SL(2,\mathbb{R})$ in \cite{sch1} and extended to $GL(2,\mathbb{R})$, $PGL(2,\mathbb{R})$ and $Sp(4,\mathbb{R})$ with maximal Toledo invariant in \cite{bs1}. With the calculation of the monodromy $\underline{\Lambda}$ in this paper one can similarly obtain the number of components of the character varieties for $SL(n,\mathbb{R}), GL(n,\mathbb{R}), PSL(n,\mathbb{R})$ and $Sp(2n,\mathbb{R})$ with maximal Toledo invariant for any value of $n$. There is however a caveat that so far for $n>2$, we are only able to count the number of connected components which intersect the regular locus. To turn these counts into a complete proof it will be necessary to show that there are no components of the corresponding character varieties lying entirely within the singular locus. We hope to address this problem in future work.\\

The strategy of counting components of $Rep(G_\mathbb{R})$ through monodromy may also be applied to non-split real forms. One may define spectral data or cameral data for arbitrary real forms and at least in some cases, this gives rise to families of smooth spectral curves or cameral curves. The generic fibre $F$ of the Hitchin fibration restricted to $\mathcal{M}_{G_\mathbb{R}}$ is generally not a discrete space anymore, but one can consider the monodromy action on $\pi_0(F)$ and by counting the number of orbits obtain a count of the components of $Rep(G_\mathbb{R})$.

\subsection{Structure of the paper}
This paper is organised as follows. In \textsection \ref{sec2} we recall the spectral curve construction of the regular locus of the Hitchin fibration. In \textsection 
\ref{secthelattices}, we study in detail the structure of the lattices $\Lambda_S , \Lambda_P$, where $\Lambda_S = H^1(S , \mathbb{Z})$ is the integral cohomology of the spectral curve over the basepoint $a_0$. In particular, we work out the intersection form $\langle \; , \; \rangle$ in Section \ref{secintf}. In Section \ref{secquadratics} we find that under conditions on $n$ and $l$, there is a natural quadratic function on $\Lambda_P$ and we compute its Arf invariant. This is an important input for the classification of the vanishing lattice $(\Lambda_P , \langle \; , \; \rangle , \Delta_P)$. In \textsection \ref{secvc} we construct vanishing cycles associated to paths $\gamma$ in $\Sigma$ joining pairs of branch points. In \textsection \ref{secthevanlat}, we use the vanishing cycles of \textsection \ref{secvc}, to construct the vanishing lattice $(\Lambda_P , \langle \; , \; \rangle , \Delta_P)$. We then recall the classification of vanishing lattices over $\mathbb{Z}_2$ in Section \ref{secz2class}, over $\mathbb{Z}$ in Section \ref{seczclass} and use this to completely determine the structure of $(\Lambda_P , \langle \; , \; \rangle , \Delta_P)$. In \textsection \ref{secmain}, we prove that the monodromy group $\Gamma_{SL}$ of the $SL(n,\mathbb{C})$-Hitchin fibration coincides with the monodromy group $\Gamma_{\Delta_P}$ of the vanishing lattice. In \textsection \ref{secappl}, we apply our monodromy computations to give a proof of Theorem \ref{thmrestriction}, describing the image of the restriction map from the cohomology of the moduli space of Higgs bundles to the cohomology of a non-singular fibre of the Hitchin fibration.

\subsection{Acknowledgements}

We would like to thank Laura Schaposnik, Tam\'as Hausel and Nigel Hitchin helpful discussions. The author is supported by the Australian Research Council grant DE160100024.

\section{Spectral curves and the regular locus}\label{sec2}

We recall some basic facts about the moduli space of $SL(n,\mathbb{C})$-Higgs bundles, in particular the spectral curve construction of the regular locus. An {\em $SL(n,\mathbb{C})$ $L$-twisted Higgs bundle} is a pair $(E,\Phi)$ where $E$ is a rank $n$ holomorphic vector bundle on $\Sigma$ with trivial determinant and $\Phi$ is a trace-free holomorphic endomorphism $\Phi : E \to E \otimes L$. As shown by Nitsure \cite{nit}, one may define notions of semistability and $S$-equivalence for twisted Higgs bundles and construct a moduli space $\mathcal{M}(n,L)$ of $S$-equivalence classes of semistable $SL(n,\mathbb{C})$ $L$-twisted Higgs bundles. The moduli space $\mathcal{M}(n,L)$ is a quasi-projective complex algebraic variety \cite{nit}. As the rank $n$ and line bundle $L$ will be fixed throughout, we will omit them from the notation and simply write $\mathcal{M}$ for the moduli space.\\

Recall the {\em Hitchin fibration}, also known as the {\em Hitchin map} or {\em Hitchin system} is a surjective holomorphic map $h : \mathcal{M} \to \mathcal{A}$ from $\mathcal{M}$ to the affine space $\mathcal{A} = \bigoplus_{j=2}^n H^0(\Sigma , L^j)$ \cite{hit1,hit2,nit}. Using the notion of spectral curves recalled below, it can be shown that the non-singular fibres of the Hitchin map are abelian varieties \cite{hit2,bnr}. Let $\mathcal{D} \subset \mathcal{A}$ denote the locus of singular fibres of the Hitchin map and let $\mathcal{A}^{\rm reg} = \mathcal{A} \setminus \mathcal{D}$ be the regular locus. We let $\mathcal{M}^{\rm reg}$ denote the points of $\mathcal{M}$ lying over the regular locus, so that $h : \mathcal{M}^{\rm reg} \to \mathcal{A}^{\rm reg}$ is a locally trivial torus bundle. Our goal in this paper is to determine the monodromy of this torus bundle.\\

We recall the spectral curve construction of $\mathcal{M}^{\rm reg}$ from $\mathcal{A}^{\rm reg}$. For this, let $tot(L)$ denote the total space of $L$ and $\pi : tot(L) \to \Sigma$ the projection. We let $\lambda \in H^0( tot(L) , \pi^*L)$ denote the tautological section of $\pi^*(L)$ on $tot(L)$. Given $a = (a_2,a_3, \dots , a_n) \in \mathcal{A}$, let $s_a$ be the section of $\pi^*L^n$ on $tot(L)$ given by $s_a = \lambda^n + a_2 \lambda^{n-2} + \dots + a_n$. The vanishing locus of $s_a$ defines a curve $S_a \subset tot(L)$ called the {\em spectral curve} associated to $a$. We use $\pi : S_a \to \Sigma$ to denote the restriction of $\pi : tot(L) \to \Sigma$. As in \cite{hit2}, Bertini's theorem implies that $S_a$ is smooth for generic $a \in \mathcal{A}$. Moreover it can be shown that $S_a$ is smooth if and only if the corresponding fibre of the Hitchin system $h : \mathcal{M} \to \mathcal{A}$ is non-singular \cite{kp}. When this is the case the corresponding fibre of the Hitchin system is:
\begin{equation*}
h^{-1}(a) = \{ M \in Jac_{ln(n-1)/2}(S_a) \; | \; Nm(M) = L^{n(n-1)/2} \},
\end{equation*}
where $Jac_{d}(S_a)$ denotes the degree $d$ component of $Pic(S_a)$ and $Nm : Pic(S) \to Pic(\Sigma)$ is the norm map associated to $\pi : S_a \to \Sigma$ \cite{bnr}. Note that $h^{-1}(a)$ is naturally a torsor over the abelian variety $Prym(S_a, \Sigma)$, which is defined as:
\begin{equation*}
Prym(S_a , \Sigma) = \{ M \in Jac(S_a) \; | \; Nm(M) = \mathcal{O} \}.
\end{equation*}
The abelian variety $Prym(S_a , \Sigma)$ is called the {\em Prym variety} associated to $\pi : S_a \to \Sigma$.\\

Let $q : \mathcal{S} \to \mathcal{A}^{\rm reg}$ denote the family of smooth spectral curves parametrised by $\mathcal{A}^{\rm reg}$. This may be defined as the set of pairs $(a,x) \in \mathcal{A}^{\rm reg} \times tot(L)$ such that $s_a(x) = 0$. This is a family of non-singular curves parametrised by $\mathcal{A}^{\rm reg}$ and so we may construct the relative Jacobian $j : Jac(\mathcal{S}/\mathcal{A}^{\rm reg}) \to \mathcal{A}^{\rm reg}$. Let $Nm : Jac(\mathcal{S}/\mathcal{A}^{\rm reg}) \to \mathcal{A}^{\rm reg} \times Jac(\Sigma)$ be the fibrewise norm map and let $p : Prym(\mathcal{S}/\mathcal{A}^{\rm reg}) \to \mathcal{A}^{\rm reg}$ be the fibrewise kernel of $Nm$. Then $Prym(\mathcal{S}/\mathcal{A}^{\rm reg})$ is a family of Prym varieties parametrised by $\mathcal{A}^{\rm reg}$.

\begin{remark}
If $n$ is odd or $l$ is even, then there exists on $\Sigma$ a square root $L^{(n-1)/2}$ of $L^{n-1}$. The map sending $a \in \mathcal{A}^{\rm reg}$ to $\pi^*( L^{(n-1)/2} ) \in Jac_{ln(n-1)/2}(S_a)$ is a section of $\mathcal{M}^{\rm reg} \to \mathcal{A}^{\rm reg}$ and this gives an isomorphism $\mathcal{M}^{\rm reg} \simeq Prym(\mathcal{S}/\mathcal{A}^{\rm reg})$.
\end{remark}

\begin{definition}
We define the following local systems on $\mathcal{A}^{\rm reg}$:
\begin{itemize}
\item[(1)]{$\underline{\Lambda} = Hom( R^1 h_* \mathbb{Z} , \mathbb{Z})$, the monodromy of the Hitchin fibration.}
\item[(2)]{$\underline{\Lambda}_J = Hom(R^1 j_* \mathbb{Z},\mathbb{Z})$, the monodromy of the family of Jacobians.}
\item[(3)]{$\underline{\Lambda}_P = Hom(R^1 p_* \mathbb{Z},\mathbb{Z})$, the monodromy of the family of Prym varieties.}
\item[(4)]{$\underline{\Lambda}_S = Hom(R^1q_*\mathbb{Z},\mathbb{Z})$, the monodromy of the family of spectral curves.}
\item[(5)]{$\underline{\Lambda}_\Sigma$ will denote the trivial local system with coefficient group $H^1(\Sigma , \mathbb{Z})$.}
\end{itemize}
\end{definition}

\begin{proposition}
We have the following relations between local systems:
\begin{itemize}
\item[(1)]{$\underline{\Lambda}_S \simeq \underline{\Lambda}_J$}
\item[(2)]{$\underline{\Lambda} \simeq \underline{\Lambda_P}$}
\item[(3)]{The fibrewise norm map $Nm : Jac(\mathcal{S}/\mathcal{A}^{\rm reg}) \to \mathcal{A}^{\rm reg} \times Jac(\Sigma)$ induces a short exact sequence:
\begin{equation}\label{equseslocsys}
0 \longrightarrow \underline{\Lambda}_P \longrightarrow \underline{\Lambda}_J \buildrel Nm \over \longrightarrow \underline{\Lambda}_\Sigma \longrightarrow 0.
\end{equation}
}
\end{itemize}

\end{proposition}
\begin{proof}
(1). For any smooth spectral curve $S$ there is a canonical isomorphism $H^1(S,\mathbb{Z}) \simeq H^1( Jac(S) , \mathbb{Z})$. Clearly this isomorphism extends to a canonical isomorphism of local systems $R^1 q_* \mathbb{Z} \simeq R^1 j_* \mathbb{Z}$ on $\mathcal{A}^{\rm reg}$ associated to the families $q : \mathcal{S} \to \mathcal{A}^{\rm reg}$ and $j : Jac(\mathcal{S}/\mathcal{A}^{\rm reg}) \to \mathcal{A}^{\rm reg}$. Thus $\underline{\Lambda}_S \simeq \underline{\Lambda}_J$.\\

(2). Recall that $Prym(\mathcal{S}/\mathcal{A}^{\rm reg})$ is the pre-image under the fibrewise norm map $Nm : Jac(\mathcal{S}/\mathcal{A}^{\rm reg}) \to \mathcal{A}^{\rm reg} \times Jac(\Sigma)$ of $\mathcal{A}^{\rm reg} \times \{ \mathcal{O} \}$. Similarly $\mathcal{M}^{\rm reg}$ is the pre-image under $Nm$ of $\mathcal{A}^{\rm reg} \times\{ L^{n(n-1)/2} \}$. Thus $p : Prym(\mathcal{S}/\mathcal{A}^{\rm reg}) \to \mathcal{A}^{\rm reg}$ is a bundle of abelian varieties and $h : \mathcal{M}^{\rm reg} \to \mathcal{A}^{\rm reg}$ is a bundle of torsors for $Prym(\mathcal{S}/\mathcal{A}^{\rm reg})$. In particular this gives a canonical isomorphism $R^1h_* \mathbb{Z} \simeq R^1p_* \mathbb{Z}$ of local systems. Hence $\underline{\Lambda} \simeq \underline{\Lambda_P}$. \\

(3). We have a short exact sequence of bundles of abelian varieties over $\mathcal{A}^{\rm reg}$:
\begin{equation*}
Prym(\mathcal{S}/\mathcal{A}^{\rm reg}) \to Jac(\mathcal{S}/\mathcal{A}^{\rm reg}) \to \mathcal{A}^{\rm reg} \times Jac(\Sigma).
\end{equation*}
Taking homology of the fibres gives the short exact sequence (\ref{equseslocsys}).
\end{proof}

Instead of working with local systems directly we will fix a basepoint $a_0 \in \mathcal{A}^{\rm reg}$ and work with representations of $\pi_1( \mathcal{A}^{\rm reg} , a_0)$. Let $\pi : S \to \Sigma$ denote the spectral curve corresponding to the basepoint $a_0$. We let $\Lambda_S$ denote $H^1(S,\mathbb{Z})$, let $\Lambda_\Sigma$ denote $H^1(\Sigma , \mathbb{Z})$ and let $\Lambda_P$ denote the kernel of the Gysin map $\pi_* : \Lambda_S \to \Lambda_\Sigma$. Then $\underline{\Lambda} \simeq \underline{\Lambda}_P$ corresponds to a representation $\rho_{SL} : \pi_1( \mathcal{A}^{\rm reg} , a_0) \to Aut( \Lambda_P)$ which we call the {\em monodromy representation of the $SL(n,\mathbb{C})$ Hitchin fibration}. Similarly $\underline{\Lambda}_S \simeq \underline{\Lambda}_J$ corresponds to a representation $\rho_{GL} : \pi_1( \mathcal{A}^{\rm reg} , a_0) \to Aut( \Lambda_S )$, which we refer to as the {\em monodromy representation of the $GL(n,\mathbb{C})$ Hitchin fibration}. We note that (\ref{equseslocsys}) corresponds to a short exact sequence of representations:
\begin{equation*}
0 \longrightarrow \Lambda_P \longrightarrow \Lambda_S \buildrel \pi_* \over \longrightarrow \Lambda_\Sigma \longrightarrow 0.
\end{equation*}

\begin{remark}
Although it will play no part in subsequent calculations, we note that the dual local system $\underline{\Lambda}^*$ and corresponding representation $\rho_{SL}^* : \pi_1(\mathcal{A}^{\rm reg} , a_0) \to Aut( \Lambda_P^*)$ give the monodromy of the Hitchin fibration for the group $PGL(n,\mathbb{C})$. This is an instance of Langlands duality for Hitchin systems, which in general implies that Langlands dual groups give rise to dual monodromy representations.
\end{remark}

\section{The lattices $\Lambda_S,\Lambda_P$}\label{secthelattices}

\subsection{Decomposition of $\Lambda_S$}\label{secdecomps}

In this section we fix a basepoint $a_0 \in \mathcal{A}^{\rm reg}$ and examine the structure of the lattices $\Lambda_S,\Lambda_P$ in detail. For this we choose $a_0$ to be of the form $a_0 = (0,0, \dots , 0 , a_n )$, for some $a_n \in H^0(\Sigma , L^n)$. The corresponding spectral curve $S$ is given by the equation $\lambda^n + a_n = 0$ and it is clear that $S$ is smooth if and only if $a_n$ has only first order zeros. Let $b_1 , b_2 , \dots , b_k \in \Sigma$ be the zeros of $a_n$, where $k = nl$ and let $u_1, \dots , u_k \in S$ be the corresponding ramification points. 

\begin{definition}
Let $D$ be the open unit disc in $\mathbb{C}$ of radius $1$ centred at $0$ and $\overline{D}$ the corresponding closed unit disc. We say that an embedding $i : \overline{D} \to \Sigma$ is a {\em trivialising disc} if $i(D)$ contains the branch points of $\pi : S \to \Sigma$ and the restriction of $S$ to $\Sigma \setminus D$ is the trivial covering space ($n$ disjoint copies of $\Sigma \setminus D$).
\end{definition}

\begin{proposition}\label{proptd}
For any $a_n \in H^0(\Sigma , L^n)$ with simple zeros, the corresponding spectral curve $\pi : S \to \Sigma$ admits a trivialising disc.
\end{proposition}
\begin{proof}
Let $\Sigma' = \Sigma \setminus \{b_1 , \dots , b_k \}$. Let $C \subset H_1( \Sigma' , \mathbb{Z})$ be the subgroup generated by cycles $l_1,l_2, \dots l_k$ around the points $b_1, \dots , b_k$, so that we have an exact sequence:
\begin{equation*}\xymatrix{
0 \ar[r] & C \ar[r] & H_1( \Sigma' , \mathbb{Z}) \ar[r] & H_1(\Sigma , \mathbb{Z}) \ar[r] & 0.
}
\end{equation*}
The monodromy of the branched cover $S \to \Sigma$ defines a homomorphism $\phi_S : H_1(\Sigma' , \mathbb{Z}) \to \mathbb{Z}_n$ such that $\phi_S(l_i) = 1$ for each $i$. Let $i : \overline{D} \to \Sigma$ be an embedding of $\overline{D}$ in $\Sigma$ for which $i(D)$ contains the branch points. Then $i$ determines a splitting $s_i : H_1(\Sigma , \mathbb{Z}) \to H_1(\Sigma' , \mathbb{Z})$ given by the composition of the isomorphism $H_1(\Sigma, \mathbb{Z}) \simeq H_1(\Sigma \setminus i(D) , \mathbb{Z})$ with the inclusion induced  map $H_1(\Sigma \setminus i(D) , \mathbb{Z}) \to H_1(\Sigma' , \mathbb{Z})$. We have that $i$ gives a trivialising disc if and only if the composition $\phi_S \circ s_i : H_1(\Sigma , \mathbb{Z}) \to \mathbb{Z}_n$ is the trivial homomorphism.\\

Choose a branch point $b_i$ and let $l$ be an embedded loop in $\Sigma'$ with underlying homology class of the form $[l] = l_i + s_i(m)$, where $m \in H_1(\Sigma , \mathbb{Z})$ is the homology class of $l$ as a loop in $\Sigma$. Let $\tau_l : \Sigma \to \Sigma$ be a Dehn twist around $l$, supported in a neighbourhood of $l$ containing no branch points. So $\tau_l$ preserves the branch points and acts a homeomorphism $\tau_l : \Sigma' \to \Sigma'$. The composition $i' = \tau_l \circ i : \overline{D} \to \Sigma$ gives a new embedding and corresponding splitting $s_{i'} : H_1(\Sigma , \mathbb{Z}) \to H_1(\Sigma' , \mathbb{Z})$. From the commutative diagram:
\begin{equation*}\xymatrix{
& & \\
H_1(\Sigma , \mathbb{Z}) \ar[d]^-{{\tau_l}_*} \ar@/^2pc/[rr]^-{s_i} & \ar[l]_-{\cong} H_1(\Sigma \setminus i(D) , \mathbb{Z}) \ar[r] \ar[d]^-{{\tau_l}_*} & H_1(\Sigma' , \mathbb{Z}) \ar[d]^-{{\tau_l}_*} \\
H_1(\Sigma , \mathbb{Z}) \ar@/_2pc/[rr]_-{s_{i'}}  & \ar[l]_-{\cong} H_1(\Sigma \setminus i'(D), \mathbb{Z}) \ar[r] & H_1(\Sigma' , \mathbb{Z}) \\
& &
}
\end{equation*}
we see that $s_{i'} = {\tau_l}_* \circ s_i \circ {\tau_l}_*^{-1}$. One finds that $({\tau_l}_* s_i {\tau_l}_*^{-1})(a) = s_i(a) + \langle a , m \rangle l_i$, hence $(\phi_S \circ s_{i'})(a) = (\phi_S \circ s_i)(a) + \langle a , m \rangle$. From this we see that by applying to $i$ a series of Dehn twists around suitably chosen loops, we can obtain an embedding $\tilde{i}$ of $\overline{D}$ with splitting $s_{\tilde{i}}$ such that $\phi_S \circ s_{\tilde{i}}$ is trivial, as required.
\end{proof}

By Proposition \ref{proptd}, there exists a trivialising disc $i : \overline{D} \to \Sigma$. Let $D'$ be an open disc obtained from $D$ by shrinking the radius slightly, but for which $i(D')$ still contains all branch points. Let $U_0 = i(D)$ and $U_1 = \Sigma \setminus \overline{i(D')}$, where $\overline{i(D')}$ is the closure of $i(D')$ in $\Sigma$. The Mayer-Vietoris sequence applied to the cover $V_0 = \pi^{-1}(U_0), V_1 = \pi^{-1}(U_1)$ of $S$ gives:
\begin{equation*}\xymatrix{
0 \ar[r] & H_2(S , \mathbb{Z} ) \ar[r]^-{\partial} & H_1(V_0 \cap V_1 , \mathbb{Z}) \ar[r]^-{({i_0}_*, {i_1}_*)} & H_1(V_0 , \mathbb{Z}) \oplus H_1( V_1 , \mathbb{Z}) \ar[r]^-{{j_0}_*-{j_1}_*} & H_1(S,\mathbb{Z}) \ar[r] & 0,
}
\end{equation*}
where $i_a$ is the inclusion $V_0 \cap V_1 \to V_a$ and $j_a$ is the inclusion $V_a \to S$. Since $i : D \to \Sigma$ is a trivialising disc we have that $V_1$ consists of $n$ disjoint copies of $U_1$ and that the inclusion $\coprod_{i=1}^n U_1 \to \coprod_{i=1}^n \Sigma$ induces an isomorphism $H_1( V_1 , \mathbb{Z}) \to \oplus_{i=1}^n H^1(\Sigma , \mathbb{Z})$. Similarly, $V_0 \cap V_1$ is homotopy equivalent to $n$ copies of the circle. These circles correspond to the $n$ lifts to $S$ of the boundary of the disc $i(D)$, hence the map ${i_1}_*$ is trivial. Let $\Lambda_{S,0}$ be the cokernel of ${i_0}_* : H_1( V_0 \cap V_1 , \mathbb{Z}) \to H_1( V_0 , \mathbb{Z})$ and $\Lambda_{S,1} = H_1( V_1 , \mathbb{Z}) = \oplus_{i=1}^n H^1(\Sigma , \mathbb{Z})$. Then the above sequence gives an isomorphism $H_1(S,\mathbb{Z}) \simeq \Lambda_{S,0} \oplus \Lambda_{S,1}$. By Poincar\'e duality we have $\Lambda_S = H^1(S,\mathbb{Z}) \simeq H_1(S,\mathbb{Z})$, so
\begin{equation}\label{equdecomps}
\Lambda_S \simeq \Lambda_{S,0} \oplus \Lambda_{S,1}.
\end{equation}
We will make extensive use of this decomposition of $\Lambda_S$ in the following sections.

\subsection{$\mathbb{Z}[t]$-module structure}\label{secztmod}

Let $t : S \to S$ be the map sending $\lambda$ to $\xi \lambda$, where $\xi = e^{2\pi i/n}$. Then $t$ generates a $\mathbb{Z}_n$-action on $S$. We can thus view $\Lambda_S$ as a $\mathbb{Z}[t]$-module, where the action of $t$ satisfies $t^n = 1$. Clearly (\ref{equdecomps}) is a direct sum of $\mathbb{Z}[t]$-modules. The lattice $\Lambda_{S,1}$ was seen to consist of a direct sum of $n$ copies of $\Lambda_\Sigma = H^1(\Sigma , \mathbb{Z}) \simeq H_1(\Sigma , \mathbb{Z})$. Further, $t$ acts on $\Lambda_{S,1}$ by cyclic permutation of these $n$ copies so that as $\mathbb{Z}[t]$-modules we have:
\begin{equation*}
\Lambda_{S,1} = \frac{\mathbb{Z}[t]}{\langle t^n - 1 \rangle} \otimes_{\mathbb{Z}} \Lambda_\Sigma.
\end{equation*}

Next consider $\Lambda_{S,0}$. From its definition $\Lambda_{S,0}$ fits into an exact sequence:
\begin{equation*}\xymatrix{
0 \ar[r] & H_2(S , \mathbb{Z} ) \ar[r]^-{\partial} & H_1(V_0 \cap V_1 , \mathbb{Z}) \ar[r]^-{ {i_0}_*} & H_1( V_0 , \mathbb{Z}) \ar[r] & \Lambda_{S,0} \ar[r] & 0.
}
\end{equation*}
Up to homotopy $V_0 \cap V_1$ can be identified with the boundary of $V_0$ and $i_0$ with the inclusion map. We have that $V_0$ is a degree $n$ branched cover of the unit disc $D \subset \mathbb{C}$. Applying a suitable homeomorphism to $D$, we can assume that the branch points $b_1,b_2, \dots , b_k$ lie on the real axis with $-1/2 = b_1 < b_2 < \dots < b_k = 1/2$. A deformation retraction of $D$ onto the interval $[-1/2 , 1/2]$ lifts to a deformation retraction of $V_0$ onto $\pi^{-1}( [-1/2,1/2])$.\\

The branched cover $\pi^{-1}( [-1/2 , 1/2]) \to [-1/2,1/2]$ is depicted in Figure \ref{figbranch}. For $1 \le i \le k-1$, let $\gamma_i$ be the path joining $b_i$ to $b_{i+1}$ along the interval $[b_i , b_{i+1}]$. Let $\gamma_i^1 , \gamma_i^2 , \dots , \gamma_i^n$ be the lifts of $\gamma_i$ to paths in $\pi^{-1}( [-1/2 , 1/2])$ joining $u_i$ to $u_{i+1}$, as shown in Figure \ref{figbranch}. We may order the lifts in such a way that $\gamma_i^j = t^{j-1}\gamma_i^1$. Let $c_i$ be the homology class of the $1$-cycle $\gamma_i^1 - \gamma_i^2$. Then $t^{j-1}c_i$ is the homology class of $\gamma_i^j - \gamma_i^{j+1}$. Note that $t^{n-1}c_i = \gamma_i^n - \gamma_i^1 = -(c_i + tc_i + \dots + t^{n-2}c_i)$, but that the cycles $c_i , tc_i , \dots , t^{n-2}c_i$ are independent. In fact, we have:
\begin{proposition}
The homology group $H_1( V_0 , \mathbb{Z})$ is free as a $\mathbb{Z}$-module, with basis given by the cycles $t^j c_i $ for $1 \le i \le k-1$, $0 \le j \le n-2$.
\end{proposition}
\begin{proof}
We have shown that $V_0$ admits a deformation retraction to $\pi^{-1}( [-1/2,1/2])$. By considering Figure \ref{figbranch}, the proposition is easily seen to follow.
\end{proof}

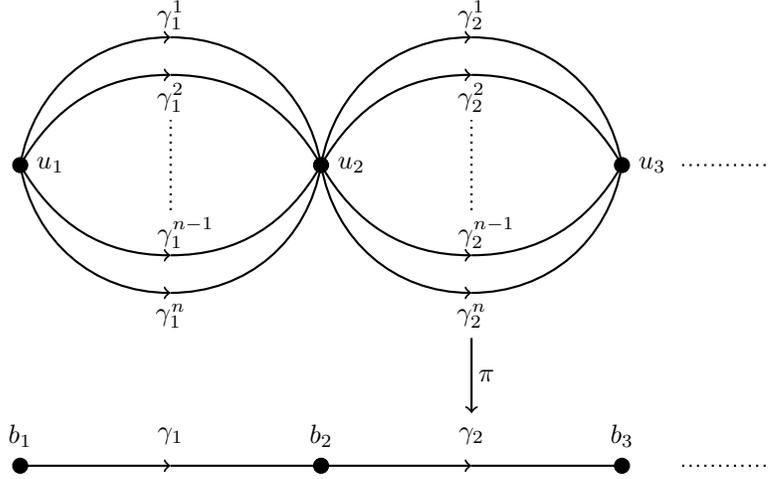
\begin{figure}
\begin{tikzpicture}
\draw [fill] (0,0) circle(0.1);
\draw [fill] (4,0) circle(0.1);
\draw [fill] (8,0) circle(0.1);

\draw [->, thick] (0,0) to [out=80, in=180] (2,1.7) ;
\draw [thick] (2,1.7) to [out=0, in=100] (4,0);
\draw [->, thick] (0,0) to [out=60, in=180] (2,1.2) ;
\draw [thick] (2,1.2) to [out=0, in=120] (4,0);
\draw [->, thick] (0,0) to [out=300, in=180] (2,-1.2) ;
\draw [thick] (2,-1.2) to [out=0, in=240] (4,0);
\draw [->, thick] (0,0) to [out=280, in=180] (2,-1.7) ;
\draw [thick] (2,-1.7) to [out=0, in=260] (4,0);
\draw[thick, dotted] (2,-0.6) -- (2,0.6) ;

\draw [->, thick] (4,0) to [out=80, in=180] (6,1.7) ;
\draw [thick] (6,1.7) to [out=0, in=100] (8,0);
\draw [->, thick] (4,0) to [out=60, in=180] (6,1.2) ;
\draw [thick] (6,1.2) to [out=0, in=120] (8,0);
\draw [->, thick] (4,0) to [out=300, in=180] (6,-1.2) ;
\draw [thick] (6,-1.2) to [out=0, in=240] (8,0);
\draw [->, thick] (4,0) to [out=280, in=180] (6,-1.7) ;
\draw [thick] (6,-1.7) to [out=0, in=260] (8,0);
\draw[thick, dotted] (6,-0.6) -- (6,0.6) ;

\node at (0.4,0) {$u_1$};
\node at (4.4,0) {$u_2$};
\node at (8.4,0) {$u_3$};

\node at (2,2) {$\gamma_1^1$};
\node at (2,0.9) {$\gamma_1^2$};
\node at (2.2,-0.9) {$\gamma_1^{n-1}$};
\node at (2,-2) {$\gamma_1^n$};

\node at (6,2) {$\gamma_2^1$};
\node at (6,0.9) {$\gamma_2^2$};
\node at (6.2,-0.9) {$\gamma_2^{n-1}$};
\node at (6,-2) {$\gamma_2^n$};

\draw[thick, dotted] (8.8,0) -- (10,0);

\draw[thick , ->] (6,-2.3) --(6,-3.3) ;
\node at (6.2,-2.8) {$\pi$};

\draw [fill] (0,-4) circle(0.1);
\draw [fill] (4,-4) circle(0.1);
\draw [fill] (8,-4) circle(0.1);

\node at (0,-3.6) {$b_1$};
\node at (4,-3.6) {$b_2$};
\node at (8,-3.6) {$b_3$};

\draw[thick, ->] (0,-4) --(2,-4);
\draw[thick] (2,-4) --(4,-4);
\draw[thick, ->] (4,-4) --(6,-4);
\draw[thick] (6,-4) --(8,-4);

\node at (2,-3.6) {$\gamma_1$};
\node at (6,-3.6) {$\gamma_2$};

\draw[thick, dotted] (8.8,-4) -- (10,-4);

\end{tikzpicture}
\caption{The branched covering $\pi^{-1}( [-1/2,1/2] ) \to [-1/2 , 1/2]$}\label{figbranch}
\end{figure}

Next, we will determine the image of ${i_0}_* : H_1( V_0 \cap V_1 , \mathbb{Z} ) \to H_1( V_0 , \mathbb{Z})$. For this we introduce a convention on the ordering of the paths $\gamma_i^j$ as follows. For each $i$, let $D_i$ be a small disc in $D$ centered around the point $b_i$. We choose these discs small enough so that they are mutually disjoint. Let $\tilde{D}_i = \pi^{-1}(D_i)$. Then $\gamma_i^1 , \gamma_i^2 , \dots , \gamma_i^{n}$ divide $\tilde{D}_i$ into $n$ segments. We will assume that the lifts $\gamma_i^j$ have been ordered in such a way that for $i < i < k$ and $1 \le j \le n-1$, we have that $\gamma_{i-1}^j \cap \pi^{-1}(D_i)$ lies in the segment between $\gamma_i^j$ and $\gamma_i^{j+1}$. By induction on $i$, such an ordering exists. Henceforth we will assume that such an ordering has been chosen. Figure \ref{figorder} shows the order of paths entering and leaving the point $u_i$ under our convention.

\begin{figure}
\begin{tikzpicture}

\path (0,0) coordinate (P);
\path (60:3) coordinate (Q);
\path (120:3) coordinate (R);

\draw [thick, ->-=.5] (P) -- (3,0);
\draw [thick, ->-=.5] (P) -- (Q);
\draw [thick, ->-=.5] (P) -- (R) ;

\draw [thick, -<-=.5] (P) -- (30:3);
\draw [thick, -<-=.5] (P) -- (90:3);
\draw [thick, -<-=.5] (P) -- (150:3);

\node at (0,-0.4) {$u_i$};
\node at (3.4,0) {$\gamma_i^1$};
\node at (60:3.4) {$\gamma_i^2$};
\node at (120:3.4) {$\gamma_i^3$};
\node at (30:3.4) {$\gamma_{i-1}^1$};
\node at (90:3.4) {$\gamma_{i-1}^2$};
\node at (150:3.4) {$\gamma_{i-1}^3$};

\draw [thick, dotted] (-2,0) arc (180:340:2);

\end{tikzpicture}
\caption{Ordering convention for the paths $\gamma_i^j$}\label{figorder}
\end{figure}
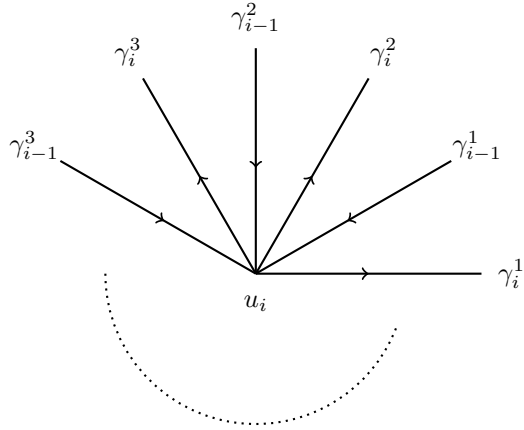

\begin{proposition}\label{propboundary}
The image of ${i_0}_* : H_1( V_0 \cap V_1 , \mathbb{Z} ) \to H_1( V_0 , \mathbb{Z})$ is the subgroup spanned by $\partial , t\partial , \dots , t^{n-1}\partial$, where

\begin{equation*}
\partial = c_1 + (1+t)c_2 + (1+t+t^2)c_3 + \dots + (1 + t + \dots + t^{n-2})c_{k-1}.
\end{equation*}
\end{proposition}
\begin{proof}
Recall that $V_0 \cap V_1$ may be identified with the boundary of $V_0$, which a disjoint union of $n$ circles. Moreover $t : S \to S$ cyclically permutes the circles. Let $\partial_D$ be a clockwise loop in $D$ enclosing the branch points $b_1 , \dots , b_k$. Let $\partial$ be any lift of $\partial_D$ to a loop in $S$. It follows that the image of ${i_0}_*$ is spanned by the cycles $\partial , t\partial , \dots , t^{n-1}\partial$. Thus it remains to determine the cycle $\partial$. We think of $\partial$ as consisting of an upper segment $\partial^+$ and a lower segment $\partial^-$.

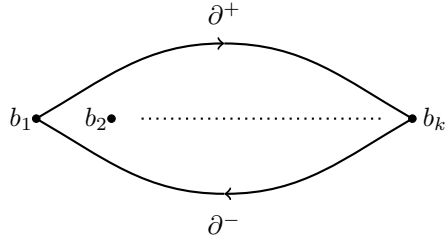
\begin{center}
\begin{tikzpicture}

\draw [fill] (0,0) circle(0.05);
\draw [fill] (1,0) circle(0.05);
\draw [fill] (5,0) circle(0.05);

\draw [thick, dotted] (1.4, 0)--(4.6, 0);

\draw [->, thick] (0,0) to [out=30, in=180] (2.5,1) ;
\draw [thick] (2.5,1) to [out=0, in=150] (5,0) ;

\draw [thick] (0,0) to [out=330, in=180] (2.5,-1) ;
\draw [<-, thick] (2.5,-1) to [out=0, in=210] (5,0) ;

\node at (2.5,1.4) {$\partial^+$};
\node at (2.5,-1.4) {$\partial^-$};

\node at (-0.2,0) {$b_1$};
\node at (0.8,0) {$b_2$};
\node at (5.3,0) {$b_k$};

\end{tikzpicture}
\end{center}

Due to our ordering conventions, it is easy to see that $\partial^+$ can be lifted to $\gamma_1^1 + \gamma_2^1 + \dots + \gamma_{k-1}^1$ and $\partial^-$ to $-\gamma_1^2 - \gamma_2^3 - \gamma_3^4 - \dots - \gamma_{k-1}^n$. We then have:
\begin{equation*}
\begin{aligned}
\partial &= (\gamma_1^1 - \gamma_1^2) + (\gamma_2^1 - \gamma_2^3) + \dots + (\gamma_{k-1}^1 - \gamma_{k-1}^n) \\
& = c_1 + (1+t)c_2 + (1+t+t^2)c_3 + \dots + (1+t + \dots + t^{n-2})c_{k-1}
\end{aligned}
\end{equation*}
as required.
\end{proof}

\begin{corollary}
As a $\mathbb{Z}[t]$-module, $\Lambda_{S,0}$ is isomorphic to $k-2$ copies of $\mathbb{Z}[t]/ \langle 1 + t + t^2 + \dots + t^{n-1} \rangle$.
\end{corollary}

We shall use the same notation $t^j c_i$ for the cycle in $H_1( V_0 , \mathbb{Z})$ and for its image in $\Lambda_{S,0}$. This should not cause confusion, since from this point onward we will only be concerned with $\Lambda_{S,0}$ as opposed to $H_1( V_0 , \mathbb{Z})$. We have that the $t^j c_i$ span $\Lambda_{S,0}$, but are not linearly independent, since by Proposition \ref{propboundary} we have the relation
\begin{equation*}
c_1 + (1+t)c_2 + (1+t+t^2)c_3 + \dots + (1+t+ \dots + t^{n-2})c_{k-1} = 0.
\end{equation*}

\subsection{Intersection form}\label{secintf}

Consider the intersection pairing $\langle \; , \; \rangle : \Lambda_S \times \Lambda_S \to \mathbb{Z}$ defined by the cup product. Clearly $\langle \; , \; \rangle$ is $t$-invariant and the decomposition $\Lambda_S = \Lambda_{S,0} \oplus \Lambda_{S,1}$ is orthogonal. Under the identification $\Lambda_{S,1} = \oplus_{i=1}^n H^1(\Sigma , \mathbb{Z})$, we have that the restriction $\langle \; , \; \rangle|_{\Lambda_{S,1}}$ is given by $n$ copies of the usual intersection form on $H^1(\Sigma , \mathbb{Z})$. The intersection form on $\Lambda_{S,0}$ is given by the following:

\begin{proposition}\label{propints}
We have the following intersection pairings:
\begin{equation*}
\begin{aligned}
\langle c_i , tc_i \rangle &= 1, &
\langle c_i , t^j c_i \rangle &= 0, \; \text{ for } 2 \le j \le n-2, \\
\langle c_i , c_{i+1} \rangle &= 1, &
\langle c_i , t^j c_{i+1} \rangle &= 0, \; \text{ for } 2 \le j \le n-2, \\
\langle c_i , tc_{i+1} \rangle &= -1, &
\langle c_i , t^j c_{i'} \rangle & = 0, \; \text{ whenever } |i-i'| > 1.
\end{aligned}
\end{equation*}

\end{proposition}

\begin{proof}
We choose representatives of the cycles $t^j c_i$ meeting transversally and directly compute their intersection. Figure \ref{figint1} shows the computation of $\langle c_1 , tc_1 \rangle = 1$ and Figure \ref{figint2} shows the computations of $\langle c_1 , c_2 \rangle = 1$ and $\langle c_1 , tc_2 \rangle = -1$. The remaining intersections are computed similarly.
\end{proof}

The intersection relations described in Proposition \ref{propints} may be visualised as a graph, shown in Figure \ref{figintgr}. Here the vertices correspond to the elements $\{ t^j c_i\}$, for $1 \le i \le k-1$ and $0 \le j \le n-2$. We draw an oriented edge from from $u$ to $v$ whenever $\langle u , v \rangle = 1$.

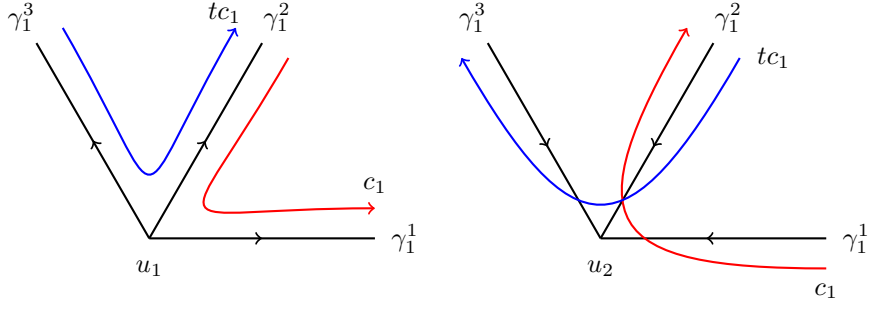
\begin{figure}
\begin{tikzpicture}
\path (0,0) coordinate (P);
\path (60:3) coordinate (Q);
\path (120:3) coordinate (R);

\draw [thick, ->-=.5] (P) -- (3,0);
\draw [thick, ->-=.5] (P) -- (Q);
\draw [thick, ->-=.5] (P) -- (R) ;

\node at (0,-0.4) {$u_1$};
\node at (3.4,0) {$\gamma_1^1$};
\node at (60:3.4) {$\gamma_1^2$};
\node at (120:3.4) {$\gamma_1^3$};

\path (3,0.4) coordinate (A);
\path (1.85,2.40) coordinate (B);
\path (0,0.4) coordinate (C);
\path (0.35,-0.2) coordinate (D);

\draw [thick, red, <-] (A) .. controls (C) and (D) .. (B);

\path (1.15,2.80) coordinate (E);
\path (-1.15,2.80) coordinate (F);
\path (-0.35,0.2) coordinate (G);
\path (0.35,0.2) coordinate (H);

\draw [thick, blue, <-] (E) .. controls (G) and (H) .. (F);

\node at (3,0.7) {$c_1$};
\node at (1,3) {$tc_1$};

\path (6,0) coordinate (S);
\path (6,0)+(60:3) coordinate (T);
\path (6,0)+(120:3) coordinate (U);

\draw [thick, -<-=.5] (S) -- (9,0);
\draw [thick, -<-=.5] (S) -- (T);
\draw [thick, -<-=.5] (S) -- (U) ;

\node at (6,-0.4) {$u_2$};
\node at (9.4,0) {$\gamma_1^1$};
\node at (7.7, 2.94) {$\gamma_1^2$};
\node at (4.3, 2.94) {$\gamma_1^3$};

\path (9,-0.4) coordinate (A1);
\path (7.15, 2.80) coordinate (B1);
\path (6,-0.4) coordinate (C1);
\path (5.65,0.2) coordinate (D1);

\draw [thick, red, ->] (A1) .. controls (C1) and (D1) .. (B1);

\path (7.85,2.40) coordinate (E1);
\path (4.15,2.40) coordinate (F1);
\path (6.35,-0.2) coordinate (G1);
\path (5.65, -0.2) coordinate (H1);

\draw [thick, blue, ->] (E1) .. controls (G1) and (H1) .. (F1);

\node at (9,-0.7) {$c_1$};
\node at (8.3,2.4) {$tc_1$};

\end{tikzpicture}
\caption{$\langle c_1 , tc_1 \rangle = 1$}\label{figint1}
\end{figure}

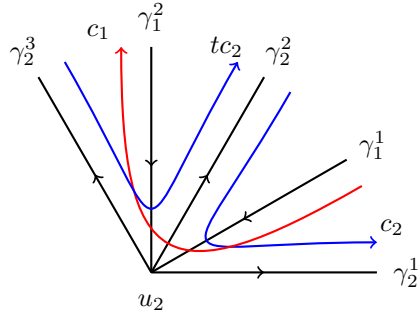
\begin{figure}
\begin{tikzpicture}
\path (0,0) coordinate (P);
\path (60:3) coordinate (Q);
\path (120:3) coordinate (R);

\path (30:3) coordinate (S);

\draw [thick, ->-=.5] (P) -- (3,0);
\draw [thick, ->-=.5] (P) -- (Q);
\draw [thick, ->-=.5] (P) -- (R) ;
\draw [thick, -<-=.5] (P) -- (S);
\draw [thick, -<-=.5] (P) -- (0,3);

\node at (0,-0.4) {$u_2$};
\node at (3.4,0) {$\gamma_2^1$};
\node at (60:3.4) {$\gamma_2^2$};
\node at (120:3.4) {$\gamma_2^3$};

\path (3,0.4) coordinate (A);
\path (1.85,2.40) coordinate (B);
\path (0,0.4) coordinate (C);
\path (0.35,-0.2) coordinate (D);

\draw [thick, blue, <-] (A) .. controls (C) and (D) .. (B);

\path (2.8,1.15) coordinate (E);
\path (-0.4,3) coordinate (F);
\path (0.2,-0.35) coordinate (G);
\path (-0.4,0) coordinate (H);

\draw [thick, red, ->] (E) .. controls (G) and (H) .. (F);

\node at (3.2,0.6) {$c_2$};
\node at (-0.7,3.2) {$c_1$};

\node at (30:3.4) {$\gamma_1^1$};
\node at (0,3.4) {$\gamma_1^2$};

\path (1.15,2.80) coordinate (E);
\path (-1.15,2.80) coordinate (F);
\path (-0.35,0.2) coordinate (G);
\path (0.35,0.2) coordinate (H);

\draw [thick, blue, <-] (E) .. controls (G) and (H) .. (F);

\node at (1,3) {$tc_2$};

\end{tikzpicture}
\caption{$\langle c_1 , c_2 \rangle = 1$, $\langle c_1 , tc_2 \rangle = -1$}\label{figint2}
\end{figure}

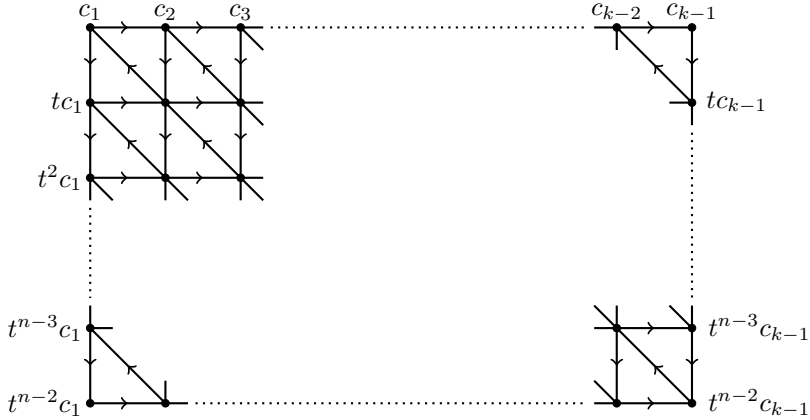
\begin{figure}
\begin{tikzpicture}

\draw [fill] (0,0) circle(0.05);
\draw [fill] (1,0) circle(0.05);
\draw [fill] (2,0) circle(0.05);
\draw [fill] (8,0) circle(0.05);
\draw [fill] (0,-1) circle(0.05);
\draw [fill] (1,-1) circle(0.05);
\draw [fill] (2,-1) circle(0.05);
\draw [fill] (8,-1) circle(0.05);
\draw [fill] (0,-2) circle(0.05);
\draw [fill] (1,-2) circle(0.05);
\draw [fill] (2,-2) circle(0.05);
\draw [fill] (0,-4) circle(0.05);
\draw [fill] (0,-5) circle(0.05);
\draw [fill] (1,-5) circle(0.05);
\draw [fill] (8,-5) circle(0.05);

\draw [fill] (8,-4) circle(0.05);
\draw [fill] (7,-5) circle(0.05);
\draw [fill] (7,-4) circle(0.05);
\draw [fill] (7,0) circle(0.05);

\draw [thick, ->-=.5] (0,0)--(1,0);
\draw [thick, ->-=.5] (1,0)--(2,0);
\draw [thick, ->-=.5] (0,-1)--(1,-1);
\draw [thick, ->-=.5] (1,-1)--(2,-1);
\draw [thick, ->-=.5] (0,-2)--(1,-2);
\draw [thick, ->-=.5] (1,-2)--(2,-2);
\draw [thick, ->-=.5] (0,0)--(0,-1);
\draw [thick, ->-=.5] (1,0)--(1,-1);
\draw [thick, ->-=.5] (2,0)--(2,-1);
\draw [thick, ->-=.5] (0,-1)--(0,-2);
\draw [thick, ->-=.5] (1,-1)--(1,-2);
\draw [thick, ->-=.5] (2,-1)--(2,-2);

\draw [thick, ->-=.5] (8,0)--(8,-1);
\draw [thick, ->-=.5] (0,-4)--(0,-5);
\draw [thick, ->-=.5] (0,-5)--(1,-5);

\draw [thick, ->-=.5] (1,-1)--(0,0);
\draw [thick, ->-=.5] (2,-1)--(1,0);
\draw [thick, ->-=.5] (1,-2)--(0,-1);
\draw [thick, ->-=.5] (2,-2)--(1,-1);
\draw [thick, ->-=.5] (1,-5)--(0,-4);

\draw [thick] (0,-2)--(0,-2.3);
\draw [thick] (1,-2)--(1,-2.3);
\draw [thick] (2,-2)--(2,-2.3);

\draw [thick] (2,0)--(2.3,0);
\draw [thick] (2,-1)--(2.3,-1);
\draw [thick] (2,-2)--(2.3,-2);

\draw [thick] (8,-1)--(8,-1.3);

\draw [thick] (0,-4)--(0,-3.7);
\draw [thick] (0,-4)--(0.3,-4);

\draw [thick] (1,-5)--(1,-4.7);
\draw [thick] (1,-5)--(1.3,-5);

\draw [thick] (2,0)--(2.3,-0.3);
\draw [thick] (2,-1)--(2.3,-1.3);
\draw [thick] (2,-2)--(2.3,-2.3);
\draw [thick] (1,-2)--(1.3,-2.3);
\draw [thick] (0,-2)--(0.3,-2.3);

\draw [thick, ->-=.5] (7,0)--(8,0);
\draw [thick, ->-=.5] (8,-1)--(7,0);

\draw [thick, ->-=.5] (7,-4)--(8,-4);
\draw [thick, ->-=.5] (7,-4)--(7,-5);
\draw [thick, ->-=.5] (8,-4)--(8,-5);
\draw [thick, ->-=.5] (7,-5)--(8,-5);
\draw [thick, ->-=.5] (8,-5)--(7,-4);

\draw [thick] (8,-4)--(8,-3.7);
\draw [thick] (8,-4)--(7.7,-3.7);

\draw [thick] (7,-4)--(7,-3.7);
\draw [thick] (7,-4)--(6.7,-3.7);
\draw [thick] (7,-4)--(6.7,-4);

\draw [thick] (7,-5)--(6.7,-4.7);
\draw [thick] (7,-5)--(6.7,-5);

\draw [thick] (7,0)--(7,-0.3);
\draw [thick] (7,0)--(6.7,0);
\draw [thick] (8,-1)--(7.7,-1);

\draw [thick, dotted] (0,-2.4)--(0,-3.6);
\draw [thick, dotted] (2.4,0)--(6.6,0);
\draw [thick, dotted] (8,-1.4)--(8,-3.6);
\draw [thick, dotted] (1.4,-5)--(6.6,-5);

\node at ( 0 , 0.2 ) {$c_1$};
\node at (-0.3 , -1 ) {$tc_1$};
\node at (-0.4 , -2 ) {$t^2c_1$};

\node at (-0.6 , -4 ) {$t^{n-3}c_1$};
\node at (-0.6 , -5 ) {$t^{n-2}c_1$};

\node at (1 , 0.2) {$c_2$};
\node at (2 , 0.2) {$c_3$};

\node at (7,0.2) {$c_{k-2}$};
\node at (8 , 0.2) {$c_{k-1}$};
\node at (8.6 , -1) {$tc_{k-1}$};
\node at (8.9 , -4) {$t^{n-3}c_{k-1}$};
\node at (8.9 , -5) {$t^{n-2}c_{k-1}$};

\end{tikzpicture}
\caption{Intersection graph for $\Lambda_{S,0}$}\label{figintgr}
\end{figure}

\subsection{Decomposition of $\Lambda_P$}\label{secdp}

Recall the decomposition $\Lambda_S = \Lambda_{S,0} \oplus \Lambda_{S,1}$. Let $i_u : \Lambda_{S,u} \to \Lambda_S$ for $u=0,1$ be the inclusions and $j_u : \Lambda_S \to \Lambda_{S,u}$ the projections. Recall also the identification $\Lambda_{S,1} = \frac{\mathbb{Z}[t]}{\langle t^n - 1 \rangle} \otimes_{\mathbb{Z}} \Lambda_\Sigma$.

\begin{proposition}\label{proppi}
The map $\pi^* : \Lambda_\Sigma \to \Lambda_S$ factors as $\pi^* = i_1 \circ \pi_1^*$, where $\pi_1^* : \Lambda_\Sigma \to \Lambda_{S,1}$ is given by $\pi_1^*(a) = (1+t+ \dots + t^{n-1})a$. The map $\pi_* : \Lambda_S \to \Lambda_\Sigma$ factors as $\pi_* = (\pi_1)_* \circ j_1$, where $(\pi_1)_* : \Lambda_{S,1} \to \Lambda_\Sigma$ is given by $(\pi_1)_*( t^j a) = a$, for $a \in \Lambda_\Sigma$.
\end{proposition}
\begin{proof}
Let $i : D \to \Sigma$ be a trivialising disc and $\Lambda_S = \Lambda_{S,0} \oplus \Lambda_{S,1}$ the corresponding decomposition. Any class in $\Lambda_\Sigma$ can be represented by a cycle lying outside of $D$ and the factorsation $\pi^* = i_1 \circ \pi_1^*$ follows. Similarly, any class in $\Lambda_{S,0}$ is represented by a cycle whose image under $\pi$ lies in $D$. So $\pi_*$ is trivial on $\Lambda_{S,0}$ and the factorisation $\pi_* = (\pi_1)_* \circ j_1$ follows.
\end{proof}

\begin{corollary}
The decomposition $\Lambda_S = \Lambda_{S,0} \oplus \Lambda_{S,1}$ induces a similar decomposition $\Lambda_P = \Lambda_{P,0} \oplus \Lambda_{P,1}$, where $\Lambda_{P,0} = \Lambda_{S,0}$ and $\Lambda_{P,1}$ is the kernel of the map $(\pi_1)_* : \Lambda_{S,1} \to \Lambda_\Sigma$ defined in Proposition \ref{proppi}. Moreover, we have an isomorphism of $\mathbb{Z}[t]$-modules $\Lambda_{P,1} \simeq \mathbb{Z}[t]/\langle 1 + t + \dots + t^{n-1} \rangle \otimes_{\mathbb{Z}} \Lambda_\Sigma$.
\end{corollary}

\begin{corollary}\label{corpolarization}
The polarization type of $\Lambda_P$ is $(1,1, \dots , 1 , n , n , \dots , n )$, where $1$ occurs $(n-2)(g-1) + n(n-1)l/2 - 1$ times and $n$ occurs $g$ times.
\end{corollary}
\begin{proof}
Since the decomposition $\Lambda_S = \Lambda_{S,0} \oplus \Lambda_{S,1}$ is orthogonal and $\langle \; , \; \rangle$ is unimodular on $\Lambda_S$, it follows that the restriction of $\langle \; , \; \rangle$ to $\Lambda_{S,0}$ or $\Lambda_{S,1}$ is unimodular. So the type of $\Lambda_{P,0} = \Lambda_{S,0}$ is $(1,1, \dots , 1)$. Choose a symplectic basis $a_1,b_1 , \dots , a_g , b_g$ of $\Lambda_\Sigma$. There is an induced decomposition of $\Lambda_{P,1}$ into $g$ orthogonal $\mathbb{Z}[t]$-submodules $M_1,M_2, \dots , M_g$, where $M_u$ is spanned by $\{ (1-t)a_u , t(1-t)a_u , \dots , t^{n-2}(1-t)a_u , (1-t)b_u , t(1-t)b_u , \dots , t^{n-2}(1-t)b_u \}$. The intersection matrix on $M_u$ can be obtained as the tensor product of the Cartan matrix for $A_{n-1}$ with the standard $2 \times 2$ symplectic matrix $\left[\begin{matrix} 0 & 1 \\ -1 & 0 \end{matrix} \right]$. It follows that the type of $M_u$ is given by the invariants of the $A_{n-1}$ Cartan matrix, which are $(1,1, \dots , 1 , n)$, where there are $n-1$ entries equal to $1$. Thus the type of $\Lambda_{P,1}$ is of the form $(1,1, \dots , n , \dots , n)$, where there are $g$ copies of $n$.
\end{proof}

In order to apply the classification of vanishing lattices in Section \ref{secthevanlat}, it will be necessary to consider cohomology with $\mathbb{Z}_2$-coefficients. Let $\Lambda_S[2] = \Lambda_S \otimes_\mathbb{Z} \mathbb{Z}_2$ be the mod $2$ reduction of $\Lambda_S$ and similarly define $\Lambda_P[2],\Lambda_\Sigma[2]$.

\begin{proposition}\label{proporthogsplit}
Suppose $n$ is odd. Then $\pi^* : \Lambda_\Sigma[2] \to \Lambda_S[2]$ gives a splitting of the sequence 
\begin{equation*}\xymatrix{
0 \ar[r] & \Lambda_P[2] \ar[r] & \Lambda_S[2] \ar[r]^-{\pi_*} & \Lambda_\Sigma[2] \ar[r] & 0.
}
\end{equation*}
The decomposition 
\begin{equation*}
\Lambda_S[2] = \Lambda_P[2] \oplus \Lambda_\Sigma[2]
\end{equation*}
induced by this splitting is orthogonal.\\

Suppose $n$ is even. Then the image of $\pi^* : \Lambda_\Sigma[2] \to \Lambda_S[2]$ is contained in $\Lambda_P[2]$. We have that $\pi^*\Lambda_\Sigma[2]$ is the null space of $\langle \; , \; \rangle|_{\Lambda_P[2]}$ and hence the quotient $H = \Lambda_P[2]/\pi^*\Lambda_\Sigma[2]$ has an induced non-degenerate alternating bilinear form $\langle \; , \; \rangle_H$. The filtration 
\begin{equation*}
\pi^* \Lambda_\Sigma[2] \subset \Lambda_P[2] \subset \Lambda_S[2]
\end{equation*}
can be split in such a way that under the induced isomorphism 
\begin{equation*}
\begin{aligned}
\Lambda_S[2] &\simeq \Lambda_\Sigma[2] \oplus (\Lambda_P[2]/\Lambda_\Sigma[2]) \oplus (\Lambda_S[2]/\Lambda_P[2]) \\
&\simeq \Lambda_\Sigma[2] \oplus H \oplus \Lambda_\Sigma[2],
\end{aligned}
\end{equation*}
we have that $\pi^*(a) = (a,0,0)$, $\pi_*(a,b,c) = c$ and 
\begin{equation*}
\langle (a,b,c) , (a',b',c') \rangle = \langle a, c' \rangle + \langle b , b' \rangle_H + \langle a' , c \rangle,
\end{equation*}
for all $a,a',c,c' \in \Lambda_\Sigma[2]$ and $b,b' \in H$.
\end{proposition}
\begin{proof}
For $a \in \Lambda_\Sigma$ we have $\pi_* \pi^* (a) = na$. Thus when $n$ is odd we have that $\pi^* : \Lambda_\Sigma[2] \to \Lambda_S[2]$ is a splitting of $\pi_* : \Lambda_S[2] \to \Lambda_\Sigma[2]$. For $a \in \Lambda_\Sigma[2]$, $b \in \Lambda_P[2]$, we have $\langle \pi^*(a) , b \rangle = \langle a , \pi_*(b) \rangle = 0$. So for $n$ odd, the splitting provided by $\pi^*$ is orthogonal. Now assume that $n$ is even. We have shown that $\pi^* \Lambda_\Sigma[2] \subset \Lambda_P[2]$. By non-degeneracy of $\langle \; , \; \rangle$ on $\Lambda_S[2]$ we have that $\Lambda_P^\perp[2] = ker(\pi_*)^\perp = im(\pi^*)$, so that $\pi^* \Lambda_\Sigma[2]$ is the nullspace of $\langle \; , \; \rangle|_{\Lambda_P[2]}$ as claimed. The quotient $H = \Lambda_P[2]/\pi^*\Lambda_\Sigma[2]$ then has an induced non-degenerate alternating bilinear form $\langle \; , \; \rangle_H$. The last claim concerning the splitting of the filtration is straightforward.
\end{proof}

\subsection{Quadratic functions on $\Lambda_S[2],\Lambda_P[2]$}\label{secquadratics}

Under conditions on $n$ and $l$ we will construct monodromy invariant quadratic functions on $\Lambda_S[2], \Lambda_P[2]$ and compute their Arf invariants.

\begin{definition}
Let $\overline{V}$ be a finite dimensional $\mathbb{Z}_2$-vector space and let $\langle \; , \; \rangle$ be a bilinear form on $\overline{V}$ which is alternating, i.e. $\langle x , x \rangle = 0$ for all $x \in \overline{V}$. A {\em quadratic function} associated to $(\overline{V} , \langle \; , \; \rangle)$ is a $\mathbb{Z}_2$-valued function $q$ on $\overline{V}$ satisfying:
\begin{equation*}
q(x+y) = q(x) + q(y) + \langle x , y \rangle
\end{equation*}
for all $x,y \in \overline{V}$.
\end{definition}
Let $\overline{V}_0 = \{ x \in \overline{V} \; | \; \langle x , \; \rangle = 0 \}$ and let $p$ be the dimension of $\overline{V}_0$. Then $\overline{V}/\overline{V}_0$ is even-dimensional as it carries a non-degenerate alternating bilinear form. Let $2r$ be the dimension of $\overline{V}/\overline{V}_0$, so that $\overline{V}$ has dimension $\mu = 2r + p$. If $q$ is a quadratic function, observe that $q|_{\overline{V}_0}$ is linear. If $q$ vanishes on $\overline{V}_0$, then $q$ descends to a quadratic function on $\overline{V}/\overline{V}_0$ and we define the {\em Arf invariant} ${\rm Arf}(q) \in \mathbb{Z}_2$ of $q$ to be $\sum_{i=1}^r q(e_i)q(f_i)$, where $e_1 , \dots , e_r , f_1 , \dots , f_r \in \overline{V}$ project to a symplectic basis of $\overline{V}/\overline{V}_0$. This can be shown to be independent of the choice of $e_1 , \dots , f_r$. If $q|_{\overline{V}_0}$ is not identically zero, then ${\rm Arf}(q)$ is left undefined. For fixed $r,p$ there are at most $3$ isomorphism classes of quadratic functions, distinguished by whether the Arf invariant is $0,1$ or undefined.\\

Let $X$ be a compact Riemann surface with canonical bundle $K_X$. Spin structures on $X$ may be identified with square roots of $K_X$. A spin structure gives a $KO$-orientation and hence a Gysin homomorphism $\varphi_X : KO(X) \to KO^{-2}(pt) = \mathbb{Z}_2$. If $E$ is a holomorphic vector bundle on $X$ with orthogonal structure, then $\varphi_X([E])$ is the mod $2$ index \cite{ati}:
\begin{equation*}
\varphi_X([E]) = dim \left( H^0( X , E \otimes K_X^{1/2} ) \right) \; ( \text{mod } 2).
\end{equation*}
Note that if $A$ is any line bundle on $X$, then $A\oplus A^{*}$ has an orthogonal structure given by pairing $A$ and $A^*$. By Riemann-Roch, we have $\varphi_X(A \oplus A^*) = deg(A) \; ( \text{mod } 2)$. The restriction of $\varphi_X$ to $H^1(X,\mathbb{Z}_2)$, the space of $\mathbb{Z}_2$-line bundles, is a quadratic function with constant term, in the sense that:
\begin{equation*}
\varphi_X( a + b ) = \varphi_X(a) + \varphi_X(b) + \langle a , b \rangle + \varphi_X(0).
\end{equation*}
In particular, if we let $\tilde{\varphi}_X = \varphi_X + \varphi_X(0)$, then $\tilde{\varphi}_X|_{H^1(X,\mathbb{Z}_2)}$ is a quadratic function. The spin structure $K_X^{1/2}$ is called {\em even} or {\em odd} according to whether $\varphi_X(0)$ is $0$ or $1$. If $X$ has genus $g_X$ then $\varphi_X$ has $2^{g_X-1}(2^{g_X}+1)$ zeros \cite{ati}. It follow that the Arf invariant of $\tilde{\varphi}|_{H^1(X,\mathbb{Z}_2)}$ equals $\varphi_X(0)$.\\

Now consider the case of a spectral curve $\pi : S \to \Sigma$ and let $K_S$ denote the canonical bundle of $S$. We will use the mod $2$ index to construct a monodromy invariant quadratic function on $\Lambda_S[2]$ provided that either $n$ is odd or $n$ and $l$ are both even. We will later see that when $n$ is even and $l$ is odd, there are no monodromy invariant quadratic functions on $\Lambda_S[2]$. By the adjunction formula we have $K_S \pi^*(K^{-1}) = \pi^*(L^{n-1})$. Under the assumption that either $n$ is odd or $n$ and $l$ are both even, there exists a square root $L^{(n-1)/2}$ of $L^{n-1}$ on $\Sigma$. This defines a relative $KO$-orientation for $\pi : S \to \Sigma$ and hence a Gysin homomorphism $\pi_{!} : KO(S) \to KO(\Sigma)$. If $E$ is a holomorphic vector bundle on $S$ with orthogonal structure, then $\pi_{!}$ is related to taking the direct image by:
\begin{equation*}
\pi_! [E] = [\pi_* (E \otimes \pi^*( L^{(n-1)/2}) )].
\end{equation*}
Note that by relative duality, $\pi_* (E \otimes \pi^*( L^{(n-1)/2}) )$ inherits an orthogonal structure from the orthogonal structure on $E$. Now choose a square root $K^{1/2}$ of $K$ and set $K_S^{1/2} = \pi^*( L^{(n-1)/2} K^{1/2})$. As we have chosen our spin structures to be compatible with the relative spin structure, we have that $\varphi_S = \varphi_\Sigma \circ \pi_!$. Moreover, since $K^{1/2}$ and $L^{(n-1)/2}$ are defined on $\Sigma$, it is clear that $\tilde{\varphi}_S |_{\Lambda_S[2]}$ is a monodromy invariant quadratic function on $\Lambda_S[2]$.

\begin{proposition}\label{proparfinvq}
Suppose that either $n$ is odd or $n$ and $l$ are even. Choose square roots of $K$ and $L^{(n-1)}$ on $\Sigma$ so that the function $\varphi_S : KO(S) \to \mathbb{Z}_2$ is defined. Then:
\begin{enumerate}
\item{The restriction of $\tilde{\varphi}_S$ to $\Lambda_P[2]$ is independent of the choice of square roots and defines a monodromy invariant quadratic function $q : \Lambda_P[2] \to \mathbb{Z}_2$.}
\item{If $n = 2m+1$ is odd, then $Arf(q) = (m(m-1)/2)l \; (\text{mod } 2)$.}
\item{If $n = 2m$ is even, then $Arf(q) = m(l/2) \; (\text{mod } 2)$.}
\end{enumerate}
\end{proposition}
\begin{proof}

To show independence of choices, suppose we replace $K^{1/2}$ and $L^{(n-1)/2}$ by $K^{1/2} \otimes A_1$ and $L^{(n-1)/2} \otimes A_2$, where $A_1,A_2 \in \Lambda_\Sigma[2]$. Then $K_S^{1/2}$ is replaced with $K_S^{1/2} \otimes \pi^*(A)$, where $A = A_1 \otimes A_2$. This has the effect of replacing the function $x \mapsto \varphi_S(x)$ with $x \mapsto \varphi_S(x+\pi^*[A])$. But if $x \in \Lambda_P[2]$ then $\langle x , \pi^*[A] \rangle = 0$, so
\begin{equation*}
\begin{aligned}
\varphi_S(x+\pi^*[A]) + \varphi_S(\pi^*[A]) &= \varphi_S(x) + \varphi_S(\pi^*[A]) + \langle x , \pi^*[A] \rangle + \varphi_S(0) + \varphi_S(\pi^*[A])\\
& = \varphi_S(x) + \varphi_S(0) \\
& = \tilde{\varphi}_S(x),
\end{aligned}
\end{equation*}
which shows independence of $\tilde{\varphi}_S$ on the choice of square roots.\\

Now suppose that $n = 2m+1$ is odd. If $A \in \Lambda_\Sigma[2]$, we find that
\begin{equation*}
\begin{aligned}
\varphi_S( \pi^*[A]) &= \varphi_\Sigma( \pi_! \pi^*[A]) \\
&= \varphi_\Sigma( [A \otimes ( L^m \oplus L^{m-1} \oplus \dots \oplus L^{-m} ) ]) \\
&= ml + (m-1)l + \dots + l + \varphi_\Sigma([A]) \\
& = (m(m-1)/2 )l + \varphi_\Sigma([A]).
\end{aligned}
\end{equation*}
In this calculation we have used that $\pi_* \mathcal{O}_S = \mathcal{O}_\Sigma \oplus L^{-1} \oplus L^{-2} \oplus \dots \oplus L^{-(n-1)}$ \cite{bnr}. It follows that $\varphi_S(0) = (m(m-1)/2)l + \varphi_\Sigma(0)$ and that $\tilde{\varphi}_S(\pi^*[A]) = \tilde{\varphi}_\Sigma([A])$. Thus the restriction of $\tilde{\varphi}_S$ to $\Lambda_\Sigma[2]$ has Arf invariant equal to $\varphi_\Sigma(0)$. Now since $n$ is odd, we have an orthogonal decomposition $\Lambda_S[2] = \Lambda_P[2] \oplus \Lambda_\Sigma[2]$ and from the additivity of the Arf invariant we have $\varphi_S(0) = Arf(q) + \varphi_\Sigma(0)$. Hence $Arf(q) = \varphi_S(0) + \varphi_\Sigma(0) = (m(m-1)/2)l$.\\

Lastly, suppose that $n$ and $l$ are even and set $n = 2m$. If $A \in \Lambda_\Sigma[2]$, we find that
\begin{equation*}
\begin{aligned}
\varphi_S( \pi^*[A]) &= \varphi_\Sigma( \pi_! \pi^*[A]) \\
&= \varphi_\Sigma( [A \otimes ( L^{(2m-1)/2} \oplus L^{(2m-3)/2} \oplus \dots \oplus L^{-(2m-1)/2} ) ]) \\
&= (2m-1)l/2 + (2m-3)l/2 + \dots + l/2. \\
&= m(l/2) \; ( \text{mod } 2).
\end{aligned}
\end{equation*}
In particular, $\varphi_S(0) = m(l/2)$ and $\tilde{\varphi}_S( \pi^*[A]) = 0$ for all $A \in \Lambda_\Sigma[2]$. By Proposition \ref{proporthogsplit}, we have $\Lambda_S[2] = \Lambda_\Sigma[2] \oplus H \oplus \Lambda_\Sigma[2]$, where $H$ is orthogonal to the two factors of $\Lambda_\Sigma[2]$. Let $J$ be the subspace of $\Lambda_S[2]$ given by $J = \Lambda_\Sigma[2] \oplus 0 \oplus \Lambda_\Sigma[2]$. Then we have an orthogonal decomposition $\Lambda_S[2] = H \oplus J$. By the above computation, $\tilde{\varphi}_S$ has at least $2^{2g}$ zeros on the subspace $J$, hence the Arf invariant of $\tilde{\varphi}_S |_J$ is $0$. This implies that $Arf( \tilde{\varphi}_S |_H ) = Arf( \tilde{\varphi}_S|_{\Lambda_S[2]} ) = \varphi_S(0) = m(l/2)$. Now recall that $\Lambda_P[2] = \pi^*\Lambda_\Sigma[2] \oplus H$, and that $\pi^*\Lambda_\Sigma[2]$ is the null space of $\langle \; , \; \rangle|_{\Lambda_P[2]}$. We have just shown that $\tilde{\varphi}_S$ vanishes on $\pi^* \Lambda_\Sigma[2]$, so the Arf invariant of $\tilde{\varphi}_S |_{\Lambda_P[2]}$ is defined and equals the Arf invariant of $\tilde{\varphi}_S |_H$, which is $m(l/2)$.
\end{proof}

\section{Constructing vanishing cycles}\label{secvc}

\subsection{General construction}
Our goal is to construct vanishing cycles associated to singular spectral curves and to show that the corresponding transvections occur in the monodromy of the Hitchin fibration. These vanishing cycles will occur as a special case of the general construction given in this section.\\

Let $X$ be a Riemann surface which may be non-compact and let $f : Y \to X$ be a degree $n$ branched cover satisfying the following conditions:
\begin{itemize}
\item[(i)]{All branch points have ramification index $n$.}
\item[(ii)]{There is an action of the cyclic group $\mathbb{Z}_n$ on $Y$ by deck transformations.}
\end{itemize}
We let $t : Y \to Y$ be a generator of the $\mathbb{Z}_n$-action. Let $b_1, b_2 , \dots , b_k$ denote the branch points of $f : Y \to X$. By assumption each branch point $b_i \in X$ has a unique ramification point $u_i \in Y$ lying over it and $f^{-1}(b_i) = u_i$.\\

Let $i \neq j$ and suppose that $\gamma : [0,1] \to X$ is an embedded path in $X$ joining $b_i$ to $b_j$ for which $\gamma(t)$ is not a branch point for any $t \in (0,1)$. Then $f^{-1}( \gamma([0,1]) )$ is the union of $n$ paths $\gamma^1 , \gamma^2 , \dots , \gamma^n : [0,1] \to Y$, each of which goes from $u_i$ to $u_j$. We may order the paths such that $t\gamma^i = \gamma^{i+1}$ for $i=1 , \dots , n-1$. We think of $\gamma^i$ as $1$-chains in $Y$ with boundary $u_j - u_i$. Thus $\gamma^1 - \gamma^2$ is a $1$-cycle in $Y$. Let $l_\gamma \in H_1(Y , \mathbb{Z})$ be the underlying homology class. Similarly $(\gamma^2 - \gamma^3) , \dots , (\gamma^{n-1} - \gamma^n) , (\gamma^n - \gamma^1)$ define cycles $tl_\gamma , \dots , t^{n-2}l_\gamma , t^{n-1}l_\gamma \in H_1(Y , \mathbb{Z})$. Note that while $l_\gamma$ depends on the choice of lift $\gamma^1$, the collection $\{ l_\gamma , tl_\gamma , t^2l_\gamma , \dots , t^{n-1}l_\gamma \}$ depends only on $\gamma$. 
\begin{definition}
We call $l_\gamma , tl_\gamma , \dots , t^{n-1}l_\gamma $ the {\em vanishing cycles associated to $\gamma$}.
\end{definition}

Let $\gamma , \gamma'$ be two embedded paths from $b_i$ to $b_j$ which avoid all other branch points. We say that $\gamma,\gamma'$ are {\em isotopic}, if they are homotopic through a path of embedded paths from $b_i$ to $b_j$ which avoid the other branch points. If $\gamma,\gamma'$ are isotopic then clearly $\gamma$ and $\gamma'$ define the same set of vanishing cycles.

\subsection{Vanishing cycles of the $A_{n-1}$ singularity}\label{secan1}
We now consider a local calculation of vanishing cycles around $A_{n-1}$ singularities. This will subsequently be converted into a global calculation for spectral curves.\\

Let $\mathbb{C}^2$ have coordinates $(\lambda,z)$ and consider the function $f : \mathbb{C}^2 \to \mathbb{C}$ given by $f(\lambda , z) = \lambda^n  +z^2$. The zero locus of $f$ is a hypersurface in $\mathbb{C}^2$ with an isolated singularity at $(0,0)$. The germ of this hypersurface around $(0,0)$ is the $A_{n-1}$ plane curve singularity. Let $\mathbb{C}^n = \mathbb{C}^2 \times \mathbb{C}^{n-2}$ have coordinates $(\lambda , z , u_1 , u_2 , \dots , u_{n-2})$ and consider the map $\tilde{f} : \mathbb{C}^n \to \mathbb{C}^{n-1}$ given by
\begin{equation*}
\tilde{f}(\lambda , z , u_1 , u_2 , \dots , u_{n-2} ) = ( \lambda^n + u_{n-2}\lambda^{n-2} + \dots + u_2 \lambda^2 + u_1 \lambda + z^2 , u_1 , u_2, \dots , u_{n-2}).
\end{equation*}
This is a versal deformation of the $A_{n-1}$ singularity. Let $B$ be a sufficiently small open ball around the origin in $\mathbb{C}^n$ with closure $\overline{B}$ and boundary $\partial B$. For any such $B$, there exists a sufficiently small open ball $B'$ around the origin in $\mathbb{C}^{n-1}$ such that $\tilde{f}$ is a submersion along $\partial B \cap \tilde{f}^{-1}(B')$. Let $D' \subset B'$ be the set of critical values of $\tilde{f}$ restricted to $\overline{B} \cap \tilde{f}^{-1}(B')$. The restriction of $\tilde{f}$ to $\overline{B} \cap \tilde{f}^{-1}( B' \setminus D')$ is a smooth fibre bundle over $B' \setminus D'$, the {\em Milnor fibration} associated to the germ of $\tilde{f}$ around $0$. Let $b \in B' \setminus D'$ be a regular value of $\tilde{f}$ in $B'$ and let $\overline{X}_b = \tilde{f}^{-1}(b) \cap \overline{B}$ be the {\em (compact) Milnor fibre}. The compact Milnor fibre is a compact manifold with boundary $\partial \overline{X}_b = \tilde{f}^{-1}(b) \cap \partial B$. The {\em geometric monodromy} \cite{loo} of the Milnor fibration is given by a representation $\rho_{geom} : \pi_1( B' \setminus D' , b ) \to Iso^0( \overline{X}_b , \partial \overline{X}_b )$, where $Iso^0( \overline{X}_b , \partial \overline{X}_b )$ is the group of relative isotopy classes of homeomorphisms of $\overline{X}_b$ which are the identity on $\partial \overline{X}_b$.\\

Let $w_0 \in \mathbb{C} \setminus \{0\}$ be small enough that $(w_0 , 0 , 0, \dots , 0) \in B'$. Then we take $b = (w_0 , 0 , \dots , 0) \in B'$ as a basepoint. The fibre of $\tilde{f}$ over $b$ is given by the equation $\lambda^n + z^2 - w_0 = 0$, which is smooth for any $w_0 \neq 0$, so $b \in B' \setminus D'$. We have:
\begin{equation*}
\overline{X}_b = \tilde{f}^{-1}(b) \cap \overline{B} = \{ (\lambda , z) \; | \; \lambda^n + z^2 - w_0 = 0, \; \; (\lambda,z ,0 , \dots , 0) \in \overline{B} \}.
\end{equation*}
Observe that $\overline{X}_b$ is a branched cover of a closed ball in $\mathbb{C}$ via the map $(\lambda , z) \mapsto z$. This is a degree $n$ cyclic branched cover with branch points $\pm\sqrt{w_0}$. As with cyclic spectral curves, we let $t : \overline{X}_b \to \overline{X}_b$ be the generator of the cyclic action given by $t(\lambda , z) = (\xi \lambda ,z)$, where $\xi = e^{2\pi i/n}$. Let $\gamma : [0,1] \to \mathbb{C}$ be the straight line in $\mathbb{C}$ joining $-\sqrt{w_0}$ to $\sqrt{w_0}$ (the choice of which square root of $w_0$ is taken to be $\sqrt{w_0}$ will be unimportant). Associated to $\gamma$ we have the vanishing cycles $l_\gamma , tl_\gamma , \dots , t^{n-1}l_\gamma$.
\begin{proposition}\label{propgeommono}
The geometric monodromy representation $\rho_{geom} : \pi_1( B' \setminus D' , b ) \to Iso^0( \overline{X}_b , \partial \overline{X}_b )$ is generated by Dehn twists of $\overline{X}_b$ around the loops $l_\gamma , tl_\gamma , \dots , t^{n-2}l_\gamma$.

\begin{remark}
Of course the Dehn twist around $t^{n-1}\gamma$ is also in the image of the geometric monodromy representation, but can be expressed in terms of $l_\gamma , \dots , t^{n-2}l_\gamma$.
\end{remark}

\end{proposition}
\begin{proof}
Recall that $D'$ is the set of critical values of $\tilde{f}$ restricted to $\overline{B} \cap \tilde{f}^{-1}(B')$. It is easy to see that $D'$ is given by:
\begin{equation*}
D' = \{ (w , u_1 , u_2 , \dots, u_{n-2} ) \in B' \; | \; \lambda^n + u_{n-2} \lambda^{n-2} + \dots + u_1 \lambda - w \text{ has multiple roots } \}.
\end{equation*}
Then $\pi_1( B' \setminus D' , b)$ is the $n$-th Artin braid group with generators $\sigma_1 , \sigma_2 , \dots , \sigma_{n-1}$ exchanging pairs of roots of $\lambda^n - w_0$. More precisely, let $q \in \mathbb{C}$ be an $n$-th root of $w_0$, so $q^n = w_0$. Then
\begin{equation*}
\lambda^n - q^n = (\lambda - q)(\lambda - \xi q) \dots (\lambda- \xi^{n-1}q).
\end{equation*}
We construct loops in $B' \setminus D'$ representing the generators $\sigma_1 , \dots , \sigma_{n-1}$ as follows. For $i = 1, \dots , n-1$, we can find a loop in $B' \setminus D'$ exchanging $\xi^{i-1} q , \xi^i q$ along paths joining $\xi^{i-1} q$ and $\xi^{i}q$ while keeping all other roots fixed. To be specific, let $\tau_i^+,\tau_i^- : [0,1] \to \mathbb{C}$ be the paths
\begin{equation*}
\begin{aligned}
\tau_i^+(t) &= \xi^{i-1}\left( 1 + \frac{1-\xi}{2}(e^{i\pi t} - 1) \right)q, \\
\tau_i^-(t) &= \xi^{i-1}\left( \xi + \frac{\xi-1}{2}(e^{i\pi t} - 1) \right)q.
\end{aligned}
\end{equation*}
Then we let $\sigma_i \in \pi_1(B' \setminus D' , b)$ be the loop in which $\xi^{i-1}q$ moves on $\tau_i^+$, $\xi^i q$ moves on $\tau_i^-$ and all other roots fixed. It follows that the geometric monodromy corresponding to $\sigma_i$ is given by a Dehn twist around a cycle $c_i$, which we now define. Let $\tau_i(t) = \xi^{i-1}\left( 1 + (\xi-1)t \right)q$ be the straight line path joining $\xi^{i-1}q$ to $\xi^{i}q$. The cycle $c_i$ may be taken to be the pre-image of $\tau_i$ in $\overline{X}_b$ under the branched double cover $(\lambda , z) \mapsto \lambda$. We have shown that $\rho_{geom}(\sigma_i)$ is a Dehn twist around the loop $c_i$. To conclude we note that it is easy to see that the cycles $c_1,c_2, \dots , c_{n-1}$ are homologous to $l_\gamma , tl_\gamma , t^{n-2}l_\gamma$ (possibly after a cyclic re-ordering of $l_\gamma , \dots , t^{n-1}l_\gamma$).
\end{proof}

\subsection{Vanishing cycles for spectral curves}

In this section we will construct a collection of vanishing cycles $\alpha \in H^1(S , \mathbb{Z})$ such that the corresponding transvections $T_\alpha$ generate the monodromy action of the Hitchin system. We continue to assume that we have chosen a basepoint $a_0 \in \mathcal{A}^{\rm reg}$ of the form $a_0 = (0, \dots , 0 , a_n)$, where $a_n$ has only simple zeros.\\

Let $i \neq j$ and suppose that $\gamma : [0,1] \to \Sigma$ is an embedded path in $\Sigma$ joining $b_i$ to $b_j$ for which $\gamma(t)$ is not a branch point for any $t \in (0,1)$. Let $l_\gamma , tl_\gamma , \dots , t^{n-1} l_\gamma \in H_1(S , \mathbb{Z})$ be the corresponding vanishing cycles. We let $c_\gamma , tc_\gamma , \dots , t^{n-1}c_\gamma \in H^1(S , \mathbb{Z})$ be the Poincar\'e dual cohomology classes. We note that $\pi_*( t^j c_\gamma) = 0$ for all $j$, so $t^j c_\gamma \in \Lambda_P$. In particular the transvection $T_{t^j c_\gamma} : \Lambda_S \to \Lambda_S$ preserves $\Lambda_P$. The main result of this section is the following:
\begin{theorem}\label{thmallowedmonodromy}
The transvections $T_{c_\gamma} , T_{tc_\gamma} , \dots , T_{t^{n-1}c_\gamma }$ belong to the $SL(n,\mathbb{C})$ monodromy group.
\end{theorem}

\begin{proof}
Let $D$ denote the unit disc in $\mathbb{C}$ and choose an embedding $e : D \to \Sigma$ such that $b_i = e( -1/2 )$, $b_j = e(1/2)$ and $\gamma(t) = e(t-1/2)$. We can also choose $D$ so that the image $e(D)$ contains no other branch points. For each $w \in D$, let $D_w$ be the degree $2$ divisor given by $D_w = e(\sqrt{w}) + e(-\sqrt{w})$. Then $D_w$ consists of two distinct points for $w \neq 0$ and $D_0 = 2 e(0)$. We also have $D_{1/4} = b_i + b_j$. Let $D^*$ be the divisor $D^* = \sum_{k \neq i,j} b_j$. Denote by $S^m \Sigma$ the $m$-th symmetric product and $\alpha : S^m \Sigma \to Jac_m(\Sigma)$ the Abel-Jacobi map, where $Jac_m(\Sigma)$ is the space of degree $m$ line bundles on $\Sigma$. We also let $\widetilde{S}^m \Sigma$ denote the subspace of $S^m \Sigma$ consisting of $m$-tuples of distinct points on $\Sigma$. For $m > 2g-2$, the restriction $\widetilde{S}^m \Sigma \to Jac_m(\Sigma)$ of the Abel-Jacobi map to $\widetilde{S}^m \Sigma$ is known to be a Serre fibration \cite{doli}. Now consider the constant map $f : D \to \widetilde{S}^{nl-2}\Sigma$ given by $f(w) = D^*$. Then $\alpha \circ f$ is the constant map $D \to Jac_{nl-2}(\Sigma)$ taking the value $\alpha(D^*) = L^n \otimes [D_{1/4}]^{-1}$. By the lifting property of Serre fibrations, $f$ is homotopic to a map $f' : D \to \widetilde{S}^{nl-2}\Sigma$ for which $\alpha(f'(w)) = L^n \otimes [D_w]^{-1}$. Here we take homotopies relative to the point $1/4 \in D$. Thus we can assume that $f'(1/4) = D^*$.\\

Consider the family of divisors $\{ D_{w} + f'(w) \}_{w \in D}$. We would like it to be true that for all $w \in D$, the divisors $D_{w}$ and $f'(w)$ have no points in common. Unfortunately this may not be the case. To avoid this, we will need to change $f'$ by a homotopy and also shrink the disc $D$ as we now explain. Define spaces $X$ and $Y$ as follows: let $X = \{ (N , w) \in \widetilde{S}^{nl-2}\Sigma \times D \; | \; [N] = L^n \otimes [D_w]^{-1} \}$. Let $Y$ be the subspace of $X$ consistsing of pairs $(N,w)$ for which some point of $N$ belongs to $D_w$. Then $X$ and $Y$ naturally fiber over $D$ giving a commutative diagram:
\begin{equation*}\xymatrix{
Y \ar[r]^-{i} \ar[dr] & X \ar[d]^-{p} \\
& D 
}
\end{equation*}
Clearly $X$ is a complex manifold and $Y$ a subvariety of codimension $1$. We also find that the projection $p : X \to D$ is a submersion. The map $f' : D \to X$ is a section of $p$. We may replace $f'$ by a homotopic map $f'' : D \to X$, which is a section of $p$ such that $f''(D)$ meets $Y$ transversally in smooth points of $Y$ and away from the fibre of $Y$ over $0 \in D$. We again choose our homotopy relative to the basepoint $1/4 \in D$, so that $f''(1/4) = D^* \notin Y$. So $Y$ meets $f''(D)$ in a discrete set of points different from $f''(1/4)$ and $f''(0)$. Choose an embedded path $p(t) : [0,1] \to D$ in $D$ from $1/4$ to $0$ with $f''( p(t) ) \notin Y$ for all $t$. Let $D'$ be a tubular neighborhood of the path $p([0,1])$ not meeting any points of $Y$. We can assume $D'$ is homeomorphic to a disc. Replacing $D$ with $D'$, we have obtained a map $f'' : D' \to X$ such that $f''(1/4) = D'$ and such that the divisor $f''(w)$ does not intersect with $D_w$ for any value of $w \in D'$.\\

Let $\sqrt{p(t)}$ denote the square root of $p(t)$ that goes from $-1/2$ to $0$. We obtain a path $\hat{p} : [0,1] \to D$ from $-1/2$ to $1/2$ by taking $\sqrt{p(t)}$ from $-1/2$ to $0$ and then taking $-\sqrt{p(t)}$ in reverse from $0$ to $1/2$. Let $\gamma' : [0,1] \to \Sigma$ be $e \circ \hat{p}$. Then $\gamma'$ is isotopic to $\gamma$ in the space of embedded paths from $b_i$ to $b_j$ avoiding all other branch points. In particular, the vanishing cycles $c_\gamma, tc_\gamma , \dots , t^{n-1}c_\gamma$ coincide with the cycles $c_{\gamma'} , tc_{\gamma'} , \dots , t^{n-1}c_{\gamma'}$.\\

Consider now the family of effective divisors $\{ D_w + f''(w) \}_{w \in D'}$ parametrised by $w \in D'$. By construction, this family has the following properties:
\begin{itemize}
\item[(i)]{For every $w$, the line bundle associate to $D_w + f''(w)$ is $L^n$.}
\item[(ii)]{When $w = 1/4$, the divisor $D_{1/4} + f''(1/4)$ is the divisor of zeros of $a_n$.}
\item[(iii)]{For $w \neq 0$, the divisor $D_w + f''(w)$ has no multiple points.}
\item[(iv)]{For $w=0$, the divisor $D_0 + f''(0)$ has one point of multiplicity $2$ and no other multiple points.}
\item[(v)]{If $w$ follows a loop based at $1/4$ going once around $0$, this has the effect of swapping $b_i, b_j$ along paths isotopic to $\gamma$ while all other branch points move on contractible loops.}
\end{itemize}
By (i), we have constructed a map $d : D' \to \mathbb{P}( H^0(\Sigma , L^n) )$. Since $D'$ is contractible this can be lifted to a map $a : D' \to H^0( \Sigma , L^n)$, where the divisor of $a(w)$ is $d(w)$. We may further choose $a$ so that $a(1/4) = a_n$. Let $a'_n = a(0) \in H^0(\Sigma , L^n)$. Then $a'_n$ has a single zero of order $2$ and no other multiple zeros. Let $x = e(0) \in \Sigma$ be the double zero of $a'_n$. Let $u \in tot(L)$ be the origin of the fibre of $L$ lying over $x$. The spectral curve $S_{a'_n}$ given by $\lambda^n + a'_n = 0$ has exactly one singularity, located at the point $u$. In a suitable local coordinate $z$ centered at $x$ we have that $a'_n(z) = z^2 (dz)^n$ and $S'_{a'_n}$ is locally given by $\lambda^n + z^2 = 0$, a plane curve singularity of type $A_{n-1}$. In what follows, we aim to show that the monodromy of this $A_{n-1}$ singularity occurs as monodromy of the Hitchin fibration and that this monodromy is generated by Picard-Lefschetz transformations in the vanishing cycles $c_\gamma , tc_\gamma , \dots , t^{n-1}c_\gamma$.\\

Let $\mathcal{A}'$ be the affine subspace of $\mathcal{A}$ consisting of points $a = (a_2 , a_3 , \dots , a_{n-1} , a'_n )$ whose $H^0(\Sigma , L^n)$-term is given by $a'_n$. For each $a \in \mathcal{A}'$, the spectral curve $S_a$ passes through the point $u \in tot(L)$ and we may consider the germ of the hypersurface $S_a \subset tot(L)$ around the point $u$. Thus $\mathcal{A}'$ gives a family of deformations of the $A_{n-1}$ singularity of $S_{a'_n}$ located at $u$. We claim this is a versal family of deformations. Under the family of deformations of the $A_{n-1}$ singularity defined by the space $\mathcal{A}'$, we have that the germ of $\lambda^n + z^2$ is deformed to $\lambda^n + a_{n-2}(z) \lambda^{n-2} + \dots + a_2(z) \lambda^2 + a_1(z) \lambda + z^2$. The Kodaira-Spencer map for this family (see \cite{loo}) is to the map $\mathcal{A}' \to \mathbb{C}^{n-2}$ sending $(a_2 , a_3 , \dots , a_{n-1} , a'_n)$ to $(a_2(x) , a_3(x) , \dots , a_{n-1}(x))$. It is clear that for any point $x \in \Sigma$, this map surjects to $\mathbb{C}^{n-2}$, which shows that $\mathcal{A}'$ provides a versal deformation of the singularity. It follows that the geometric monodromy of Section \ref{secan1} is realised as the monodromy of the Hitchin fibration around certain loops in $\mathcal{A}^{\rm reg}$, defined in the vicinity of $a'_n$. These loops act on $S$ by Dehn twists around the vanishing cycles of this $A_{n-1}$ singularity. Clearly a Dehn twist around a vanishing cycle $v \in H^1(S,\mathbb{Z})$ acts on cohomology by the corresponding Picard-Lesfschetz transformation $T_v$. To complete the proof of the theorem, it remains to identify the vanishing cycles associated to the $A_{n-1}$ singularity of $S_{a'_n}$ with specific cohomology classes in $H^1(S,\mathbb{Z})$.\\

Let $B$ be a small open ball around $u \in tot(L)$. Let $U \subset \mathcal{A}$ be an open neighbourhood of $(0,0, \dots , 0 , a'_n)$ in $\mathcal{A}$ sufficiently small so that for any $a \in U$, the spectral curve $S_a$ is smooth outside of $B$. Then after possibly further shrinking $U$, we know that the monodromy of the Hitchin fibration associated to loops in $U \setminus \mathcal{D} \cap U$ is generated by transvections of the vanishing cycles associated to the $A_{n-1}$ singularity of $S_{a'_n}$. Let $t_0 \in [0,1)$ be such that $a'_0 = a( p(t_0)) \in U$. Let $\gamma'_0$ be an embedded path in $e(D)$ joining $e(\sqrt{p(t_0)})$ to $e(-\sqrt{p(t_0)})$, in fact we may take $\gamma'_0(t)$ to be the restriction of $\gamma'$ to $[t_0/2 , 1 - t_0/2]$. Then by Proposition \ref{propgeommono}, the monodromy of the Hitchin fibration over $U \setminus \mathcal{D} \cap U$ is generated by transvections in the vanishing cycles $c_{\gamma'_0} , \dots , t^{n-1} c_{\gamma'_0} \in H^1( S_{a'_0} , \mathbb{Z})$ associated to $\gamma'_0$. Consider the path in $\mathcal{A}^{\rm reg}$ joining $a_n$ to $a'_0$ given by restricting $a( p(t) )$ to $[0,t_0]$. The Gauss-Manin connection over this path defines an isomorphism $H^1( S_{a'_0} , \mathbb{Z}) \simeq H^1( S , \mathbb{Z})$ and it is easy to see that the vanishing cycles $c_{\gamma'_0} , tc_{\gamma'_0} , \dots , t^{n-1}c_{\gamma'_0} \in H^1(S_{a'_0} , \mathbb{Z})$ associated to $\gamma'_0$ are mapped to the vanishing cycles $c_{\gamma'} , tc_{\gamma'} , \dots , t^{n-1}c_{\gamma'} \in H^1(S , \mathbb{Z})$ associated to $\gamma'$. This proves the theorem.
\end{proof}

\begin{remark}
In the process of proving Theorem \ref{thmallowedmonodromy} we constructed a family $a : D' \to H^0(\Sigma , L^n)$ of sections of $L^n$ parametrised by a space $D'$ homeomorphic to a disc. Moreover, $D'$ contained exactly one point $0 \in D'$ for which the corresponding spectral curve was singular. If we take a loop in $D'$ that winds once anti-clockwise around $0$, the corresponding monodromy transformation is seen to be $T_{c_\gamma} T_{tc_\gamma} T_{t^2 c_\gamma} \dots T_{t^{n-2}c_\gamma}$.
\end{remark}

Next, we will apply Theorem \ref{thmallowedmonodromy} to construct some specific examples of vanishing cycles. In fact we will later see that the vanishing cycles constructed below already are sufficient to generate the monodromy group. Recall from Section \ref{secdecomps} that on choosing a trivialising disc $i : \overline{D} \to \Sigma$, we obtain a decomposition $\Lambda_S = \Lambda_{S,0} \oplus \Lambda_{S,1}$, where $\Lambda_{S,0}$ is spanned by the cycles $\{ t^j c_i \}_{i,j}$ and $\Lambda_{S,1}$ can be identified with $\mathbb{Z}[t]/\langle t^n-1 \rangle \otimes_{\mathbb{Z}} \Lambda_\Sigma$. We note that once a choice of trivialising disc is given, the identification 
\begin{equation}\label{identification}
\Lambda_{S,1} \simeq \mathbb{Z}[t]/\langle t^n-1 \rangle \otimes_{\mathbb{Z}} \Lambda_\Sigma
\end{equation}
is determined only up to composition with a power of $t$.

\begin{proposition}\label{propallowedcycles}
For $i = 1,2, \dots , k-1$, there exists a path from $b_i$ to $b_{i+1}$ whose associated vanishing cycles are $\{ c_i , tc_i , \dots , t^{n-1}c_i\}$. Let $a_1, b_1 , \dots , a_g , b_g$ be a symplectic basis for $\Lambda_\Sigma$. For $u = 1,2, \dots , g$, there exist paths $\gamma_{a_u}$ and $\gamma_{b_u}$ from $b_1$ to $b_2$ such that under a suitable choice of identification (\ref{identification}), the associated vanishing cycles are $\{ c_1 + (1-t)a_u , t(c_1 + (1-t)a_u) , \dots , t^{n-1}(c_1 + (1-t)a_u) \}$ and $\{ c_1 + (1-t)b_u , t(c_1 + (1-t)b_u) , \dots , t^{n-1}(c_1 + (1-t)b_u) \}$.
\end{proposition}
\begin{proof}
Let $\gamma_i$ be the path from $b_i$ to $b_{i+1}$ constructed in Section \ref{secztmod} ($\gamma_i$ is depicted in Figure \ref{figbranch}). Then by the definition of $c_i , tc_i , \dots , t^{n-1}c_i$, it is clear that these are the vanishing cycles associated to $\gamma_i$. Consider a path $\gamma_{a_1}$ in $\Sigma$ starting at $b_1$ and moving left to the boundary of the trivialising disc, going around the loop $a_1$, then returning back to the trivialising disc and terminating at $b_2$. The corresponding vanishing cycles will clearly have the form $t^j( c_1 + (1-t)t^w a_1)$, for some value of $w$. Similarly define paths $\gamma_{b_1} , \dots , \gamma_{a_g} , \gamma_{b_g}$. The corresponding vanishing cycles will have the form $t^j( c_1 + (1-t)t^w a_u)$ or $t^j( c_1 + (1-t)t^w b_u)$ for the same value of $w$. Choosing a different identification in (\ref{identification}) if necessary, we may assume $w=0$.
\end{proof}

\section{The vanishing lattice}\label{secthevanlat}

\subsection{Construction of vanishing lattice on $\Lambda_P$}\label{secvanlat}

In this section we give $\Lambda_P$ the structure of a vanishing lattice $(\Lambda_P , \langle \; , \; \rangle , \Delta_P)$, made out of vanishing cycles constructed in Section \ref{secvc}. We then recall the classification of vanishing lattices in \cite{jan1,jan2} and use these results to classify the vanishing lattice on $\Lambda_P$. In Section \ref{secmain} we will show that the group generated by transvections in $\Delta_P$ is precisely the monodromy group $\Gamma_{SL}$ of the Hitchin fibration.\\

Vanishing lattices were classified for $R = \mathbb{Z}_2$ in \cite{jan1} and $R = \mathbb{Z}$ in \cite{jan2}. If $(V , \langle \; , \; \rangle , \Delta)$ is a vanishing lattice over $\mathbb{Z}$ then we obtain a vanishing lattice $(\overline{V} , \langle \; , \; \rangle , \overline{\Delta})$ over $\mathbb{Z}_2$, where $\overline{V} = V \otimes_{\mathbb{Z}} \mathbb{Z}_2$ and $\overline{\Delta}$ is the image of $\Delta$ under mod $2$ reduction.

\begin{proposition}\label{propsgen}
Let $V$ be a free $R$-module of finite rank with alternating bilinear form $\langle \; , \; \rangle$. Let $S \subseteq V$ be a subset of $V$ satisfying the following conditions:
\begin{enumerate}
\item[(i)]{$S$ spans $V$.}
\item[(ii)]{For any two distinct elements $u,v \in S$, there exists a sequence $u = s_1 , s_2 , \dots , s_m = v$ of elements of $S$ such that $\langle s_i , s_{i+1} \rangle = \pm 1$ for, $1 \le i \le m-1$.}
\end{enumerate}
Let $\Gamma_S$ be the subgroup of $Sp^{\#}V$ generated by $\{ T_s \}_{s \in S}$ and set $\Delta = \Gamma_S \cdot S$. Then $(V , \langle \; , \; \rangle , \Delta)$ is a vanishing lattice and $\Gamma_{\Delta} = \Gamma_S$.
\end{proposition}
\begin{proof}
Clearly $\Delta$ spans $V$, as $S$ spans $V$. If $V$ has rank $\mu > 1$ then by (i), $S$ has more than one element. Hence by (ii) there exists $\delta_1,\delta_2 \in S \subseteq \Delta$ with $\langle \delta_1 , \delta_2 \rangle = 1$. Next, we claim that for any two elements $u,v \in S$, we have $u = g(v)$ for some $g \in \Gamma_S$. By (ii) it suffices to show this in the case that $\langle u , v \rangle = \pm 1$. The claim follows, as $u = T_v T_u(v)$, if $\langle u , v \rangle = 1$ and $u = T_v^{-1} T_u^{-1}(v)$, if $\langle u , v \rangle = -1$. Our claim shows that $\Delta$ is a $\Gamma_S$-orbit, hence also a $\Gamma_{\Delta}$-orbit. Thus $(V,\langle \; , \; \rangle , \Delta)$ is a vanishing lattice. Moreover, since $\Delta$ is a $\Gamma_S$-orbit, we have that for any $\alpha \in \Delta$ there exists $u \in S$ and $g \in \Gamma_S$ with $\alpha = g(u)$. Then $T_\alpha = g \circ T_u \circ g^{-1} \in \Gamma_S$, hence $\Gamma_\Delta = \Gamma_S$.
\end{proof}

Consider $\Lambda_P$ equipped with the intersection form $\langle \; , \; \rangle$. We will give $(\Lambda_P , \langle \; , \; \rangle)$ the structure of a vanishing lattice. For this let $\{ a_u , b_u \}_{u=1}^g$ be a symplectic basis for $\Lambda_\Sigma$. Let $S_P \subset \Lambda_P$ be the subset:
\begin{equation*}
S_P = \{ t^j c_i \}_{\substack{ 0 \le j \le n-2 \\ 1 \le i \le nl-1 }} \cup \{ t^j( c_1 + (1-t)a_u) \}_{\substack{ 0 \le j \le n-2 \\ 1 \le u \le g }} \cup \{ t^j( c_1 + (1-t)b_u) \}_{\substack{ 0 \le j \le n-2 \\ 1 \le u \le g }}
\end{equation*}
By Proposition \ref{propallowedcycles}, we see that elements of $S_P$ are vanishing cycles associated to certain paths in $\Sigma$. By Theorem \ref{thmallowedmonodromy}, it follows that the transvections $\{ T_v \}_{v \in S_P}$ belong to the $SL(n,\mathbb{C})$-monodromy group $\Gamma_{SL}$. Clearly $S_P$ satisfies the conditions of Proposition \ref{propsgen}, so we obtain a vanishing lattice $(\Lambda_P , \langle \; , \; \rangle , \Delta_P )$, where $\Gamma_{\Delta_P}$ is the group generated by transvections in $S_P$ and $\Delta_P = \Gamma_{\Delta_P} \cdot S_P$. So $\Gamma_{\Delta_P}$ is a subgroup of $\Gamma_{SL}$. We will eventually show that $\Gamma_{\Delta_P} = \Gamma_{SL}$.

\begin{remark}
Let $\mathcal{VC} \subset \Lambda_P$ be the set of all vanishing cycles associated to paths $\gamma$ joining pairs of branch points in $\Sigma$ and let $\Gamma_{\mathcal{VC}}$ be the group generated by transvections in $\mathcal{VC}$. We will eventually be able to show that $\Delta_P = \Gamma_{\mathcal{VC}} \cdot \mathcal{VC}$ and $\Gamma_{\mathcal{VC}} = \Gamma_{\Delta_P}$, so that $(\Lambda_P , \langle \; , \; \rangle , \Delta_P)$ is the same as the vanishing lattice described in the introduction.
\end{remark}

\subsection{Classification over $\mathbb{Z}_2$}\label{secz2class}

We recall the classification of vanishing lattices in \cite{jan1,jan2}. Since the classification over $\mathbb{Z}$ depends on the classification over $\mathbb{Z}_2$, we begin with the $\mathbb{Z}_2$ case. In this section $\overline{V}$ will be a finite dimensional $\mathbb{Z}_2$ vector space of dimension $\mu$, equipped with an alternating bilinear form $\langle \; , \; \rangle$. Let $\overline{V}_0$ be the null space of $\langle \; , \; \rangle$ and let $p$ be the dimension of $\overline{V}_0$. Then $\mu = 2r+p$ for some integer $r$.\\

If $q$ is a quadratic function on $\overline{V}$, we see that the transvection $T_v$ associated to $v \in \overline{V}$ preserves $q$ if and only if $q(v) = 1$. We observe that to any basis $B$ of $\overline{V}$, we can associate a unique quadratic function $q_B$ with the property that $q_B(v) = 1$ for all $v \in B$. Clearly the group generated by transvections by elements of $B$ preserves $q_B$.

\begin{definition}
Let $(\overline{V} , \langle \; , \; \rangle , \overline{\Delta} )$ be a vanishing lattice. A basis $B$ of $\overline{V}$ is called {\em weakly distinguished} if $\Gamma_{\overline{\Delta}}$ is generated by $\{ T_v \}_{v \in B}$. In this case, $\Gamma_{\overline{\Delta}}$ preserves $q_B$, hence $q_B(\alpha)=1$ for all $\alpha \in \overline{\Delta}$.
\end{definition}

Note that a vanishing lattice does not necessarily admit a weakly distinguished basis.\\

To any basis $B$ of $\overline{V}$, we construct a graph $Gr(B)$ as follows. The vertices of $Gr(B)$ are elements of $B$ and for every distinct pair of element $u,v \in B$, there is a single edge joining $u$ and $v$ if and only if $\langle u , v \rangle = 1$. 
\begin{remark}\label{remwdb}
If $(\overline{V} , \langle \; , \; \rangle , \overline{\Delta} )$ is a vanishing lattice and $B$ a weakly distinguished basis consisting of elements of $\overline{\Delta}$, then $(\overline{V} , \langle \; , \; \rangle , \overline{\Delta})$ can be completely recovered from $Gr(B)$. Indeed, $\Gamma_{\overline{\Delta}}$ is the group generated by $\{ T_v \}_{v \in B}$ and $\overline{\Delta} = \Gamma_{\overline{\Delta}} \cdot B$. Note however that different graphs can give rise to the same underlying vanishing lattice. 
\end{remark}

We now state the classification in \cite{jan1} of vanishing lattices over $\mathbb{Z}_2$. If $(\overline{V}, \langle \; , \; \rangle , \overline{\Delta})$ is a vanishing lattice which admits a weakly distinguished basis $B$ whose elements belong to $\overline{\Delta}$, then by Remark \ref{remwdb} it is enough to simply give the graph $Gr(B)$. The remaining cases will be described individually. Vanishing lattices over $\mathbb{Z}_2$ can be broadly classified into three main types: {\em symplectic, orthogonal} and {\em special}. In both the orthogonal and special cases, the are three sub-cases to consider.\\

{\bf Case 1: Symplectic.} In this case $\overline{\Delta} = \overline{V} \setminus \overline{V}_0$ and $\Gamma_{\overline{\Delta}} = Sp^{\#}V$. Symplectic vanishing lattices do not admit weakly distinguished bases and the group $Sp^{\#}V$ does not preserve any quadratic functions. For fixed values of $(r,p)$ there is only one symplectic vanishing lattice, which is denoted by $Sp^{\#}(2r,p)$.\\

{\bf Case 2: Orthogonal.} In this case, a weakly distinguished basis $B$ exists, $\overline{\Delta} = \{ v \in \overline{V} \setminus \overline{V}_0 \; | \; q_B(v) = 1 \}$ and $\Gamma_{\overline{\Delta}} = O^{\#}(q_B)$ is the subgroup of $Sp^{\#}V$ preserving $q_B$. There are three sub-cases which are distinguished according to whether the Arf invariant of $q_B$ is $0,1$ or undefined. In all cases, a weakly distinguished basis can be chosen so that $Gr(B)$ is one of the following:

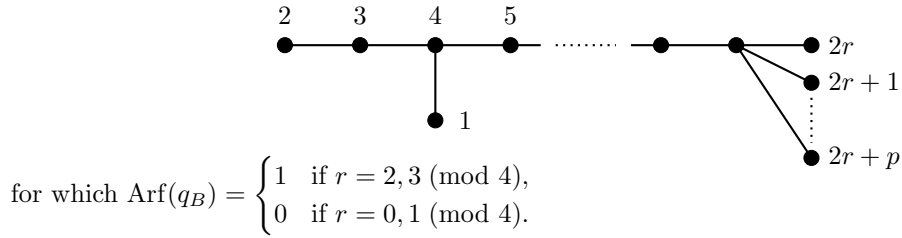
\begin{center}
\begin{tikzpicture}
\draw [thick] (0,0) -- (3.4,0) ;
\draw [thick] (4.6,0) -- (7,0) ;
\draw [thick] (2,0) -- (2,-1) ;
\draw [dotted, thick] (3.6,0) -- (4.4,0);
\draw [fill] (0,0) circle(0.1);
\draw [fill] (1,0) circle(0.1);
\draw [fill] (2,0) circle(0.1);
\draw [fill] (3,0) circle(0.1);
\draw [fill] (2,-1) circle(0.1);
\draw [fill] (5,0) circle(0.1);
\draw [fill] (6,0) circle(0.1);
\draw [fill] (7,0) circle(0.1);
\draw [fill] (7,-0.5) circle(0.1);
\draw [fill] (7,-1.5) circle(0.1);
\node at (0,0.4) {$2$};
\node at (1,0.4) {$3$};
\node at (2,0.4) {$4$};
\node at (3,0.4) {$5$};
\node at (2.4,-1) {$1$};
\node at (7.4,0) {$2r$};
\node at (7.7,-0.5) {$2r+1$};
\node at (7.7,-1.5) {$2r+p$};
\draw [thick] (6,0) -- (7,-0.5) ;
\draw [dotted, thick] (7,-0.7) -- (7,-1.3) ;
\draw [thick] (6,0) -- (7,-1.5) ;

\node at (0 , -2) {for which ${\rm Arf}(q_B) = \begin{cases} 1 \; \; \text{ if } r = 2,3 \;( \text{mod }4), \\ 0 \;\; \text{ if } r = 0,1 \;( \text{mod }4). \end{cases}$} ;

\end{tikzpicture}
\end{center}

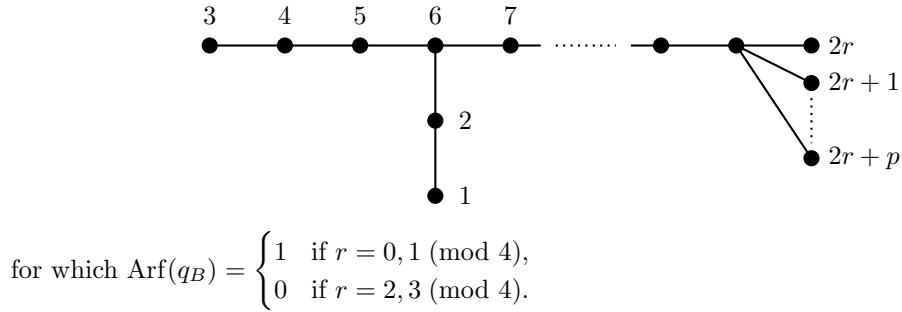
\begin{center}
\begin{tikzpicture}
\draw [thick] (-1,0) -- (3.4,0) ;
\draw [thick] (4.6,0) -- (7,0) ;
\draw [thick] (2,0) -- (2,-2) ;
\draw [dotted, thick] (3.6,0) -- (4.4,0);
\draw [fill] (2,-2) circle(0.1);
\draw [fill] (-1,0) circle(0.1);
\draw [fill] (0,0) circle(0.1);
\draw [fill] (1,0) circle(0.1);
\draw [fill] (2,0) circle(0.1);
\draw [fill] (3,0) circle(0.1);
\draw [fill] (2,-1) circle(0.1);
\draw [fill] (5,0) circle(0.1);
\draw [fill] (6,0) circle(0.1);
\draw [fill] (7,0) circle(0.1);
\draw [fill] (7,-0.5) circle(0.1);
\draw [fill] (7,-1.5) circle(0.1);
\node at (-1,0.4) {$3$};
\node at (0,0.4) {$4$};
\node at (1,0.4) {$5$};
\node at (2,0.4) {$6$};
\node at (3,0.4) {$7$};
\node at (2.4,-2) {$1$};
\node at (2.4,-1) {$2$};
\node at (7.4,0) {$2r$};
\node at (7.7,-0.5) {$2r+1$};
\node at (7.7,-1.5) {$2r+p$};
\draw [thick] (6,0) -- (7,-0.5) ;
\draw [dotted, thick] (7,-0.7) -- (7,-1.3) ;
\draw [thick] (6,0) -- (7,-1.5) ;

\node at (0,-3) {for which ${\rm Arf}(q_B) = \begin{cases} 1 \; \; \text{ if } r = 0,1 \;( \text{mod }4), \\ 0 \;\; \text{ if } r = 2,3 \;( \text{mod }4). \end{cases}$};

\end{tikzpicture}
\end{center}

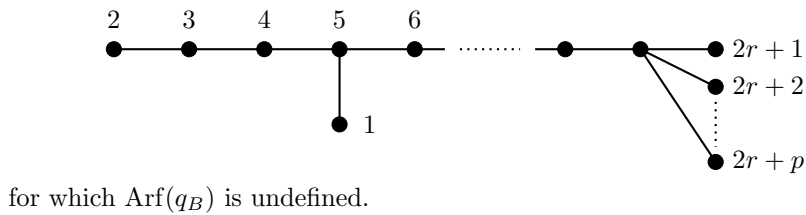
\begin{center}
\begin{tikzpicture}
\draw [thick] (-1,0) -- (3.4,0) ;
\draw [thick] (4.6,0) -- (7,0) ;
\draw [thick] (2,0) -- (2,-1) ;
\draw [dotted, thick] (3.6,0) -- (4.4,0);
\draw [fill] (-1,0) circle(0.1);
\draw [fill] (0,0) circle(0.1);
\draw [fill] (1,0) circle(0.1);
\draw [fill] (2,0) circle(0.1);
\draw [fill] (3,0) circle(0.1);
\draw [fill] (2,-1) circle(0.1);
\draw [fill] (5,0) circle(0.1);
\draw [fill] (6,0) circle(0.1);
\draw [fill] (7,0) circle(0.1);
\draw [fill] (7,-0.5) circle(0.1);
\draw [fill] (7,-1.5) circle(0.1);
\node at (-1,0.4) {$2$};
\node at (0,0.4) {$3$};
\node at (1,0.4) {$4$};
\node at (2,0.4) {$5$};
\node at (3,0.4) {$6$};
\node at (2.4,-1) {$1$};
\node at (7.7,0) {$2r+1$};
\node at (7.7,-0.5) {$2r+2$};
\node at (7.7,-1.5) {$2r+p$};
\draw [thick] (6,0) -- (7,-0.5) ;
\draw [dotted, thick] (7,-0.7) -- (7,-1.3) ;
\draw [thick] (6,0) -- (7,-1.5) ;

\node at (0,-2) {for which ${\rm Arf}(q_B)$ is undefined.};

\end{tikzpicture}
\end{center}

For fixed $(r,p)$ there are at most three orthogonal vanishing lattices, according to whether the Arf invariant is $0$, $1$ or undefined. We denote these by $O^{\#}_0(2r,p), O^{\#}_1(2r,p)$ and $O^{\#}(2r,p)$. Note that $O^{\#}(2r,p)$ only exists for $p > 0$.\\

{\bf Case 3: Special.} There are three sub-cases, of which two admit weakly distinguished bases. We describe these cases first. In the first subcase $B$ can be chosen so that $Gr(B)$ is:

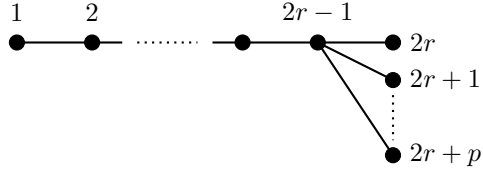
\begin{center}
\begin{tikzpicture}
\draw [thick] (2,0) -- (3.4,0) ;
\draw [thick] (4.6,0) -- (7,0) ;
\draw [dotted, thick] (3.6,0) -- (4.4,0);
\draw [fill] (2,0) circle(0.1);
\draw [fill] (3,0) circle(0.1);
\draw [fill] (5,0) circle(0.1);
\draw [fill] (6,0) circle(0.1);
\draw [fill] (7,0) circle(0.1);
\draw [fill] (7,-0.5) circle(0.1);
\draw [fill] (7,-1.5) circle(0.1);
\node at (2,0.4) {$1$};
\node at (3,0.4) {$2$};
\node at (6,0.4) {$2r-1$};
\node at (7.4,0) {$2r$};
\node at (7.7,-0.5) {$2r+1$};
\node at (7.7,-1.5) {$2r+p$};
\draw [thick] (6,0) -- (7,-0.5) ;
\draw [dotted, thick] (7,-0.7) -- (7,-1.3) ;
\draw [thick] (6,0) -- (7,-1.5) ;
\end{tikzpicture}
\end{center}

We have ${\rm Arf}(q_B) = \begin{cases} 1 \; \; \text{ if } r = 1,2 \;( \text{mod }4), \\ 0 \;\; \text{ if } r = 0,3 \;( \text{mod }4). \end{cases}$ This vanishing lattice is denoted $A^{ev}(2r,p)$.\\

In the second subcase, $Gr(B)$ is:

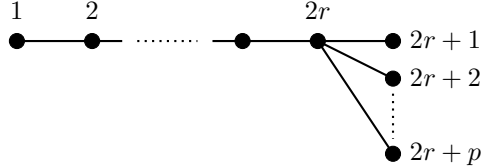
\begin{center}
\begin{tikzpicture}
\draw [thick] (2,0) -- (3.4,0) ;
\draw [thick] (4.6,0) -- (7,0) ;
\draw [dotted, thick] (3.6,0) -- (4.4,0);
\draw [fill] (2,0) circle(0.1);
\draw [fill] (3,0) circle(0.1);
\draw [fill] (5,0) circle(0.1);
\draw [fill] (6,0) circle(0.1);
\draw [fill] (7,0) circle(0.1);
\draw [fill] (7,-0.5) circle(0.1);
\draw [fill] (7,-1.5) circle(0.1);
\node at (2,0.4) {$1$};
\node at (3,0.4) {$2$};
\node at (6,0.4) {$2r$};
\node at (7.7,0) {$2r+1$};
\node at (7.7,-0.5) {$2r+2$};
\node at (7.7,-1.5) {$2r+p$};
\draw [thick] (6,0) -- (7,-0.5) ;
\draw [dotted, thick] (7,-0.7) -- (7,-1.3) ;
\draw [thick] (6,0) -- (7,-1.5) ;
\end{tikzpicture}
\end{center}

We have ${\rm Arf}(q_B) = \begin{cases} 1 \; \; \text{ if } r = 1 \;( \text{mod }4), \\ 0 \;\; \text{ if } r = 3 \;( \text{mod }4), \\ \text{undefined} \; \; \text{ if } r = 2,4 \;( \text{mod }4). \end{cases}$ This vanishing lattice is denoted $A^{odd}(2r,p)$.\\

The third subcase is obtained by taking the vanishing lattice $A^{odd}(2r,p+1)$ and taking the quotient by the subspace spanned by $e_1 + e_3 + e_5 + \dots + e_{2r+1}$, where $e_i$ denotes the basis vector corresponding to the $i$-th vertex. The resulting vanishing lattice is denoted by $A'(2r,p)$. Note that the quadratic form of the $A^{odd}(2r,p+1)$ vanishing lattice descends to the quotient if and only if $r$ is odd.\\

The following is a very useful criterion for determining whether or not a vanishing lattice is special:
\begin{proposition}[\cite{jan1}, Proposition 4.13]\label{propnotspecial}
Let $(\overline{V} , \langle \; , \; \rangle , \overline{\Delta})$ be a vanishing lattice over $\mathbb{Z}_2$. The following are equivalent:
\begin{itemize}
\item[(1)]{$(\overline{V} , \langle \; , \; \rangle , \overline{\Delta})$ is not of special type.}
\item[(2)]{There exists $\overline{V}_1 \subseteq \overline{V}$, $\overline{\Delta}_1 \subseteq \overline{\Delta}$ such that $(\overline{V}_1 , \langle \; , \; \rangle , \overline{\Delta}_1)$ is a vanishing lattice of type $O^{\#}_1(6,0)$.}
\end{itemize}
\end{proposition}

The condition that $(\overline{V} , \langle \; , \; \rangle , \overline{\Delta})$ contains $O^{\#}_1(6,0)$ can be re-stated more simply as the condition that there exists $ e_1,e_2,e_3,e_4,e_5,e_6 \in \overline{\Delta}$ such that the graph of $\{ e_1,e_2,e_3,e_4,e_5,e_6\}$ is the $E_6$ Dynkin diagram:

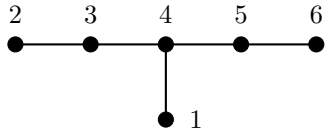
\begin{center}
\begin{tikzpicture}
\draw [thick] (0,0) -- (4,0) ;
\draw [thick] (2,0) -- (2,-1) ;

\draw [fill] (0,0) circle(0.1);
\draw [fill] (1,0) circle(0.1);
\draw [fill] (2,0) circle(0.1);
\draw [fill] (3,0) circle(0.1);
\draw [fill] (2,-1) circle(0.1);
\draw [fill] (4,0) circle(0.1);

\node at (0,0.4) {$2$};
\node at (1,0.4) {$3$};
\node at (2,0.4) {$4$};
\node at (3,0.4) {$5$};
\node at (2.4,-1) {$1$};
\node at (4,0.4) {$6$};

\end{tikzpicture}
\end{center}

Let $(\Lambda_P , \langle \; , \rangle , \Delta_P)$ be the vanishing lattice constructed in Section \ref{secvanlat} and let $(\Lambda_P[2] , \langle \; , \; \rangle , \overline{\Delta}_P)$ be its mod $2$ reduction.

\begin{theorem}\label{thmz2class}
Let $\mu = (n-1)(nl + 2g-2)$, let $p = 2g$ if $n$ is even and $p=0$ if $n$ is odd. Further define $r$ such that $\mu = 2r + p$. The vanishing lattice $(\Lambda_P[2] , \langle \; , \; \rangle , \overline{\Delta}_P)$ is isomorphic to:
\begin{enumerate}
\item{$A'(2l-2,2g)$, if $n=2$,}
\item{$O_a^{\#}(2r,0)$, where $a = (m(m-1)/2)l$, if $n =2m+1$ is odd,}
\item{$O_a^{\#}(2r,2g)$, where $a = m(l/2)$, if $n=2m$ and $l$ are even and $n>2$,}
\item{$Sp^{\#}(2r,2g)$, if $n$ is even, $l$ is odd and $n > 2$.}
\end{enumerate}
\end{theorem}
\begin{proof}
Consider first the case that $n = 2$. Then
\begin{equation*}
S_P = \{ c_1 , c_2,  \dots , c_{2l-1} , c_1 + (1-t)a_1 , c_1 + (1-t)b_1 , \dots , c_1+(1-t)a_g , c_1 + (1-t)b_g \}.
\end{equation*}
The intersection graph of $S_P$ is given by:

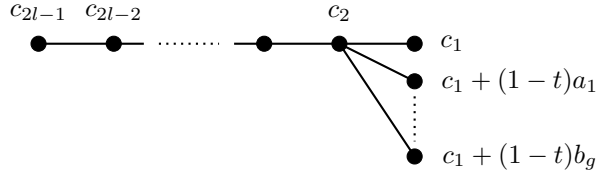
\begin{center}
\begin{tikzpicture}
\draw [thick] (2,0) -- (3.4,0) ;
\draw [thick] (4.6,0) -- (7,0) ;
\draw [dotted, thick] (3.6,0) -- (4.4,0);
\draw [fill] (2,0) circle(0.1);
\draw [fill] (3,0) circle(0.1);
\draw [fill] (5,0) circle(0.1);
\draw [fill] (6,0) circle(0.1);
\draw [fill] (7,0) circle(0.1);
\draw [fill] (7,-0.5) circle(0.1);
\draw [fill] (7,-1.5) circle(0.1);
\node at (2,0.4) {$c_{2l-1}$};
\node at (3,0.4) {$c_{2l-2}$};
\node at (6,0.4) {$c_2$};
\node at (7.5,0) {$c_1$};
\node at (8.4,-0.5) {$c_1 + (1-t)a_1$};
\node at (8.4,-1.5) {$c_1 + (1-t)b_g$};
\draw [thick] (6,0) -- (7,-0.5) ;
\draw [dotted, thick] (7,-0.7) -- (7,-1.3) ;
\draw [thick] (6,0) -- (7,-1.5) ;
\end{tikzpicture}
\end{center}

So $(\Lambda_P[2] , \langle \; , \; \rangle , \overline{\Delta}_P)$ must either be of type $A^{odd}(2l-2,2g+1)$ or $A'(2l-2,2g)$. However, we have the relation $c_1 + c_3 + c_5 + \dots + c_{2l-1} = 0$, so the vanishing lattice is of type $A'(2l-2,2g)$.\\

Consider the case $n > 2$. We will show that $(\Lambda_P[2] , \langle \; , \; \rangle , \overline{\Delta}_P)$ contains a copy of $O^{\#}_1(6,0)$. In fact, consider the elements $tc_3 , tc_1 , c_2+tc_2 , c_3 , c_4 , c_5 \in \overline{\Delta}_P$. Note that $c_2+tc_2 \in \overline{\Delta}_P$ because $c_2 + tc_2 = T_{c_2}(tc_2)$. The intersection graph of these elements is the $E_6$ Dynkin diagram:

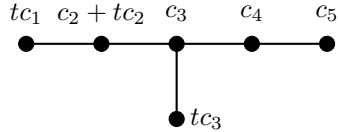
\begin{center}
\begin{tikzpicture}
\draw [thick] (0,0) -- (4,0) ;
\draw [thick] (2,0) -- (2,-1) ;

\draw [fill] (0,0) circle(0.1);
\draw [fill] (1,0) circle(0.1);
\draw [fill] (2,0) circle(0.1);
\draw [fill] (3,0) circle(0.1);
\draw [fill] (2,-1) circle(0.1);
\draw [fill] (4,0) circle(0.1);

\node at (0,0.4) {$tc_1$};
\node at (1,0.4) {$c_2+tc_2$};
\node at (2,0.4) {$c_3$};
\node at (3,0.4) {$c_4$};
\node at (2.4,-1) {$tc_3$};
\node at (4,0.4) {$c_5$};

\end{tikzpicture}
\end{center}

Hence by Proposition \ref{propnotspecial}, the vanishing lattice $(\Lambda_P[2] , \langle \; , \; \rangle , \overline{\Delta}_P)$ is not special for $n > 2$. If $n$ is odd or $n$ and $l$ are even, we saw in Section \ref{secquadratics}, that $\Lambda_P[2]$ has a monodromy invariant quadratic function $q$ and we calculated the Arf invariant of $q$ in Proposition \ref{proparfinvq}. Thus if $n$ is odd or $n$ and $l$ are both even, then $(\Lambda_P[2] , \langle \; , \; \rangle , \overline{\Delta}_P)$ is of orthogonal type with the specified Arf invariant.\\

Lastly, in the case that $n > 2$ is even and $l$ is odd we will show that $(\Lambda_P[2] , \langle \; , \; \rangle , \overline{\Delta}_P)$ is not of orthogonal type, hence it must be symplectic. Suppose on the contrary that there is a quadratic function $q : \Lambda_P[2] \to \mathbb{Z}_2$ invariant under $\Gamma_{\overline{\Delta}_P}$. Then we must have $q( t^j c_i ) = 1$ for all $i$ and $j$. From this it follows that $q( c_1 + (1+t)c_2 + \dots + (1+t+ \dots+ t^{n-2})c_{k-1}) = 1$. But this is impossible, since $c_1 + (1+t)c_2 + \dots + (1+t+\dots+t^{n-2})c_{k-1} = 0$.
\end{proof}


\subsection{Classification over $\mathbb{Z}$}\label{seczclass}

Let $(V , \langle \; , \; \rangle , \Delta)$ be a vanishing lattice over $\mathbb{Z}$. Recall \cite{bou} that for an alternating bilinear form $\langle \; , \; \rangle$ on $V$ there exists a basis $\{ e_1 , f_1 , \dots , e_r , f_r , g_1 , \dots , g_p \}$ of $V$ for which the matrix of $\langle \; , \; \rangle$ has the form
\begin{equation*}
\left[\begin{matrix} 0 & d_1 & & & & & & & & \\
-d_1 & 0 & & & & & & & & \\
& & 0 & d_2 & & & & & & \\
& & -d_2 & 0 & & & & & & \\
& & & & \ddots & & & & & \\
& & & & & 0 & d_r & & & \\
& & & & & -d_r & 0 & & & \\
& & & & & & & 0 & & \\
& & & & & & & & \ddots & \\
& & & & & & & & & 0
\end{matrix}\right]
\end{equation*} 
where the $d_i$ are positive integers and $d_i$ divides $d_{i+1}$ ($i=1, \dots , r-1$). The $d_i$ are called the elementary divisors of $\langle \; , \; \rangle$ and are uniquely determined.\\

Let $(\overline{V} , \langle \; , \; \rangle , \overline{\Delta})$ be the mod $2$ reduction, let $\overline{V}_0$ be the null space of $\langle \; , \; \rangle$ on $\overline{V}$. Choose an element $\delta \in \overline{\Delta}$ and let $\overline{V}_{00} = \{ v \in \overline{V}_0 \; | \; v + \delta \in \overline{\Delta} \}$. Then $\overline{V}_{00}$ is a subspace of $\overline{V}_0$ and does not depend on the choice of $\delta$ \cite[Lemma 2.11]{jan1}. Let $j : V \to V^*$ be the homomorphism $j(x) = \langle x , \; \rangle$ and consider the homomorphism
\begin{equation*}
j^{-1}(2V^*) \to j^{-1}(2V^*)/2V \simeq \overline{V}_0 \to \overline{V}_0/\overline{V}_{00}.
\end{equation*}
In \cite{jan1} it is shown that $\overline{V}_0/\overline{V}_{00}$ is either $0$ or $\mathbb{Z}_2$. Hence we obtain a homomorphism $\phi : j^{-1}(2V^*) \to \mathbb{Z}_2$, where in the case $\overline{V}_0/\overline{V}_{00} = 0$, we take $\phi = 0$. Define
\begin{equation*}
k_0(V) = \max \{ k \; | \; \phi( j^{-1}(2^k V^*)) \neq 0 \},
\end{equation*}
with the conventions that $k_0(V) = \infty$ if no such maximum value of $k$ exists and $k_0(V) = 0$ if $\phi = 0$. Then we have:
\begin{theorem}[\cite{jan2}, Theorem 7.5]
Let $(V_1 , \langle \; , \; \rangle , \Delta_1)$ and $(V_2 , \langle \; , \; \rangle , \Delta_2)$ be vanishing lattices over $\mathbb{Z}$. They are isomorphic if and only if (i) they are isomorphic as lattices with bilinear form, (ii) their mod $2$ reductions are isomorphic as vanishing lattices and (iii) $k_0(V_1) = k_0(V_2)$.
\end{theorem}
Thus integral vanishing lattices are classified by their mod $2$ reduction, the invariants $d_1,d_2, \dots , d_r , p$ of the bilinear form and the invariant $k_0$. Moreover, we have that $\overline{V}_{00} = \overline{V}_0$ except in the $O^{\#}$ and $A^{odd}$ cases. Following Janssen, we use notation like $O^{\#}(d_1,\dots , d_r ; p ; k_0)$ to denote a vanishing lattice with mod $2$ reduction of type $O^{\#}$ and invariants $d_1, \dots , d_r , p , k_0$. In all cases other than $O^{\#}$ and $A^{odd}$ we may omit $k_0$ from the notation. We will denote by $u$ the number of $d_i$'s which are odd.

\begin{theorem}[\cite{jan2}, Theorem 7.8]
Up to isomorphism integral vanishing lattices are given by the following list:
\begin{equation*}
\begin{aligned}
O^{\#}_1(d_1, \dots , d_r ; p) \\
O^{\#}_0(d_1, \dots , d_r ; p) && (u \ge 3) \\
O^{\#}(d_1 , \dots , d_r ; p ; k_0) && (u \ge 2, r > u \text{ or } p>0, k_0 > 0) \\
Sp^{\#}(d_1 , \dots , d_r ; p)\\
A^{ev}(d_1 , \dots , d_r ; p) && (\text{if } u=1, \text{ then } r=1 \text{ and } p=0) \\
A^{odd}(d_1, \dots, d_r ; p ; k_0) && (r>u \text{ or } p>0; \; k_0 = 0 \text{ iff } u=1) \\
A'(d_1 , \dots , d_r ; p) && (u \ge 2).
\end{aligned}
\end{equation*}

\end{theorem}

Now we can determine the isomorphism class of the integral vanishing lattice $(\Lambda_P , \langle \; , \; \rangle , \Delta_P )$:
\begin{theorem}
The vanishing lattice $(\Lambda_P , \langle \; , \; \rangle , \Delta_P)$ is isomorphic to:
\begin{enumerate}
\item{$A'(1,1, \dots , 1 , 2 , 2 , \dots , 2; 0)$, if $n=2$,}
\item{$O_a^{\#}(1,1, \dots , 1 , n , n , \dots , n ; 0)$, where $a = (m(m-1)/2)l$, if $n =2m+1$ is odd,}
\item{$O_a^{\#}(1,1, \dots , 1 , n , n , \dots , n ; 0)$, where $a = m(l/2)$, if $n=2m$ and $l$ are even and $n>2$,}
\item{$Sp^{\#}(1,1, \dots , 1 , n , n , \dots , n ; 0)$, if $n$ is even, $l$ is odd and $n > 2$.}
\end{enumerate}
In this classification, the number of $1$'s is $(n-2)(g-1) + n(n-1)l/2 - 1$ and the number of $n$'s is $g$.
\end{theorem}
\begin{proof}
The mod $2$ classification of $(\Lambda_P , \langle \; , \; \rangle , \Delta_P)$ was given in Theorem \ref{thmz2class}. In particular we saw that $(\Lambda_P , \langle \; , \; \rangle , \Delta_P)$ does not have type $O^{\#}$ or $A^{odd}$, so we do not need to consider the invariant $k_0$. The intersection form $\langle \; , \; \rangle$ is non-degenerate on $\Lambda_P$, so $p=0$ and the invariants $(d_1 , \dots , d_r)$ are precisely the polarization type of the Prym variety $Prym(S,\Sigma)$. This was calculated in Corollary \ref{corpolarization}.
\end{proof}

\section{Proof of the main theorem}\label{secmain}

Let $\delta$ be a smooth point of the discriminant locus $\mathcal{D}$. Let $D_\delta$ be a copy of the unit disc in $\mathbb{C}$, embedded in $\mathcal{A}$ such that $D_\delta$ intersects $\mathcal{D}$ transversally at the point $\delta$ and meets no other point of $\mathcal{D}$. Consider a loop $\gamma_\delta$ in $\mathcal{A} \setminus \mathcal{D}$, which starts at the basepoint $a_0$, follows a path $p$ from $a_0$ to the boundary of $D_\delta$, goes once around the boundary of $D_\delta$ and goes back to $a_0$ along $p^{-1}$. Such a loop will be called a {\em meridian}.\\ 

Kouvidakis and Pantev showed in the case $L = K$, that the discriminant locus $\mathcal{D} \subseteq \mathcal{A}$ is an irreducible hypersurface \cite{kp}. Their proof easily extends to the case where $L \neq K$ and $deg(L) > deg(K)$. It follows that $\pi_1( \mathcal{A} \setminus{D} , a_0)$ is generated by meridians and since $\mathcal{D}$ is irreducible, any two meridians are conjugate in $\pi_1(\mathcal{A}\setminus{D} , a_0)$. Let $\mathcal{D}^0 \subset \mathcal{D}$ be the locus of points $a \in \mathcal{D}$ for which the corresponding spectral curve $S_a$ is irreducible and has an ordinary double point as its only singularity. Kouvidakis and Pantev also showed that $\mathcal{D}^0$ is a non-empty Zariski open subset of $\mathcal{D}$, in the case $L=K$ \cite{kp}. This is clearly also true in the case $L \neq K$ and $deg(L) > deg(K)$. Since $\mathcal{D}^0$ is Zariski dense in $\mathcal{D}$, we see that $\pi_1(\mathcal{A} \setminus{D} , a_0)$ is generated by meridians around points in $\mathcal{D}^0$.\\

Let $\delta \in \mathcal{D}^0$ and suppose that $D_\delta$ is an embedded disc meeting $\mathcal{D}$ transversally in $\delta$, as above. Consider the family $X \to D_\delta$ of spectral curves over $D_\delta$ defined by pullback under the incusion $D_\delta \subset \mathcal{A}$. The fibre $X_\delta$ over $\delta$ of this family has an isolated non-degenerate singularity. It follows that the monodromy $\rho( \gamma_\delta) \in Aut( \Lambda_S )$ of a meridian $\gamma_\delta$ around $\delta$ is a Picard-Lefschetz transformation:
\begin{equation*}
\rho(\gamma_\delta)(x) = T_\alpha(x) = x + \langle \alpha , x \rangle \alpha,
\end{equation*}
where $\alpha \in \Lambda_S$ is the vanishing cycle associated to $\gamma_\delta$. In fact, since $\pi_* \circ \rho(\gamma_\delta) = \pi_*$, we see that $\alpha \in \Lambda_P$. We then have:
\begin{lemma}\label{lemgenvan}
Let $\Delta \subseteq \Lambda_P$ be the set of vanishing cycles associated to meridians around points in $\mathcal{D}^0$. The monodromy group $\Gamma_{SL}$ is generated by the transvections $\{ T_\alpha \}_{\alpha \in \Delta}$. For any two vanishing cycles $\alpha , \beta \in \Delta$, we have $\alpha = g \beta$ or $\alpha = -g \beta$ for some $g \in \Gamma_{SL}$.
\end{lemma}
\begin{proof}
Only the last statement of the lemma requires proof. Let $\alpha , \beta \in \Delta$. Then $T_\alpha = \rho(\gamma_1) , T_\beta = \rho(\gamma_2)$ for some meridians $\gamma_1,\gamma_2$. But we have seen that all meridians are conjugate, so there exists $x \in \pi_1( \mathcal{A} \setminus {D} , a_0)$ for which $\gamma_1 = x \gamma_2 x^{-1}$. Applying $\rho$, we get $T_\alpha = g T_\beta g^{-1} = T_{g \beta }$, where $g = \rho(x) \in \Gamma_{SL}$. As $\langle \; , \; \rangle$ is non-degenerate it is easy to see that $T_\alpha = T_{g \beta}$ implies that $\alpha = g\beta$ or $\alpha = -g\beta$.
\end{proof}

\begin{lemma}[\cite{jan1}, Theorem 2.9]\label{lemdeltaz}
Let $(V , \langle \; , \; \rangle , \Delta)$ be an integral vanishing lattice and $x \in V$. Then $x \in \Delta$ if and only of there exists $y \in V$ and $\delta \in \Delta$ such that $\langle x , y  \rangle = 1$ and $x - \delta \in 2V$.
\end{lemma}

\begin{theorem}\label{thmmain1}
Let $(\Lambda_P , \langle \; , \; \rangle , \Delta_P)$ be the integral vanishing lattice constructed in Section \ref{secvanlat} and $\Gamma_{\Delta_P} \subseteq Aut(\Lambda_P)$ the group generated by transvections by elements of $\Delta_P$. We have an equality $\Gamma_{SL} = \Gamma_{\Delta_P}$.
\end{theorem}
\begin{proof}
We have already established the inclusion $\Gamma_{\Delta_P} \subseteq \Gamma_{SL}$, which follows by Theorem \ref{thmallowedmonodromy} and Proposition \ref{propallowedcycles}. It remains to prove the reverse inclusion.\\

Consider first the case $n > 2$. By Lemma \ref{lemgenvan}, it is sufficient to show that $T_\alpha \in \Gamma_{\Delta_P}$, where $\alpha$ is the vanishing cycle of a meridian around a point in $\mathcal{D}^0$. It is easy to see that the vanishing cycles $\{ t^j c_i \}$ constructed in Section \ref{secvc} are vanishing cycles associated to meridians. Applying Lemma \ref{lemgenvan}, we have that there exists a $g \in \Gamma_{SL}$ for which $\alpha = gc_1$ or $\alpha = -gc_1$. As $(\Lambda_P , \langle \; , \; \rangle , \Delta_P)$ is a vanishing lattice, it can be shown that there exists $h \in \Gamma_{\Delta_P}$ such that $hc_1 = -c_1$ \cite{jan1}. So replacing $g$ by $gh$ if necessary, we can assume $\alpha = gc_1$. If $n$ is odd or $n$ and $l$ are both even, then $(\Lambda_P , \langle \; , \; \rangle , \Delta_P)$ is of orthogonal type. Let $\overline{\Delta}_P \subset \Lambda_P[2]$ be the mod $2$ reduction of $\Delta_P$. Then as explained in Section \ref{secz2class}, $\overline{\Delta}_P = \{ v \in \overline{V} \setminus \overline{V}_0 \; | \; q(v) = 1 \}$, where $\overline{V} = \Lambda_P[2]$ and $q$ is the monodromy invariant quadratic function. It follows that the mod $2$ reduction of $\alpha$ belongs to $\overline{\Delta}_P$, since $q(gc_1) = q(c_1) = 1$. Moreover $\langle \alpha , y \rangle = 1$, where $y = g(tc_1)$, so by Lemma \ref{lemdeltaz} we have $\alpha \in \Delta_P$. If $n$ is even and $l$ is odd, then $(\Lambda_P , \langle \; , \; \rangle , \Delta_P)$ is of symplectic type. So $\overline{\Delta}_P = \overline{V} \setminus \overline{V}_0$ and by a similar argument we have $\alpha \in \Delta_P$.\\

Lastly, consider the case $n=2$. It was shown in \cite{bs1} that $\Gamma_{SL}$ is generated by transvections $T_{c_\gamma}$, where $c_\gamma$ is the cycle associated to a path $\gamma$ joining two branch points, as in Theorem \ref{thmallowedmonodromy}. It is also clear that $c_\gamma$ is the vanishing cycle of a meridian. Suppose that $\gamma$ joins branch points $b_i,b_j$ and assume $i < j$ (the case $i>j$ is similar). Let $\overline{c} \in \Lambda_P[2]$ be the mod $2$ reduction of $c_\gamma$. Then it is easy to see that $\overline{c}$ must have the form $\overline{c} = c_i + c_{i+1} + \dots + c_{j-1} + (1+t)a$ for some $a \in \Lambda_\Sigma[2]$. In the case $n=2$, we have that $\overline{\Delta}_P \subset \Lambda_P[2]$ is of special type. Then using \cite[Lemma 3.11]{jan1}, we see that $\overline{c} \in \overline{\Delta}_P$. By the same argument as in the $n > 2$ case, we have $c_\gamma = g(c_1)$ for some $g \in \Gamma_{SL}$. So $\langle c_\gamma , y \rangle = 1$ for $y = g(c_2)$. Then by Lemma \ref{lemdeltaz}, we have $c_\gamma \in \Delta_P$.
\end{proof}

The next theorem shows that the monodromy of the $GL(n,\mathbb{C})$ Hitchin fibration is determined by the vanishing lattice $(\Lambda_P , \langle \; , \; \rangle , \Delta_P)$ together with the extension $\Lambda_P \to \Lambda_S \to \Lambda_\Sigma$.

\begin{theorem}
The monodromy group $\Gamma_{GL}$ is the subgroup of $Aut(\Lambda_S)$ generated by transvections $T_{\alpha} : \Lambda_S \to \Lambda_S$, where $\alpha \in \Delta_P$.
\end{theorem}
\begin{proof}
By the same argument used for $\Gamma_{SL}$, we have that $\Gamma_{GL}$ is the subgroup of $Aut(\Lambda_S)$ generated by transvections in vanishing cycles associated to meridians around points in $\mathcal{D}^0$. Let $\Gamma_{GL,\Delta_P}$ be the subgroup of $Aut(\Lambda_S)$ generated by transvections in $\Delta_P$. Recall from Section \ref{secvanlat} the subset $S_P \subset \Delta_P$. We let $\Gamma_{GL,S_P}$ be the subgroup of $Aut(\Lambda_S)$ generated by transvections in $S_P$.\\

In the proof of Theorem \ref{thmmain1}, we established that if $\alpha$ is such a vanishing cycle, then $\alpha \in \Delta_P$. Hence $\Gamma_{GL} \subseteq \Gamma_{GL,\Delta_P}$. On the other hand we have that the elements of $S_P$ are vanishing cycles associated to meridians, so $\Gamma_{GL,S_P} \subseteq \Gamma_{GL}$. Arguing as in the proof of Proposition \ref{propsgen}, we see that the subgroup of $Aut(\Lambda_S)$ generated by transvections in $\Delta_P$ is also generated by transvections in $S_P$. Hence $\Gamma_{GL,S_P} = \Gamma_{GL} = \Gamma_{GL,\Delta_P}$. In particular, $\Gamma_{GL}$ is the subgroup of $Aut(\Lambda_S)$ generated by transvections in $\Delta_P$.
\end{proof}

\begin{proposition}
Let $\mathcal{VC} \subset \Lambda_P$ be the set of vanishing cycles associated to paths in $\Sigma$ joining pairs of branch points and let $\Gamma_{\mathcal{VC}}$ be the group generated by transvections by cycles in $\mathcal{VC}$. Then $\Delta_P = \Gamma_{\mathcal{VC}} \cdot \mathcal{VC}$ and $\Gamma_{\mathcal{VC}} = \Gamma_{\Delta_P} = \Gamma_{SL}$.
\end{proposition}
\begin{proof}
The elements of $S_P$ are all vanishing cycles associated to certain paths joining branch points, so $S_P \subseteq \mathcal{VC}$ and $\Gamma_{S_P} \subseteq \Gamma_{\mathcal{VC}}$, where $\Gamma_{S_P}$ is the group generated by transvections in $S_P$. By Proposition \ref{propsgen}, we have $\Gamma_{S_P} = \Gamma_{\Delta_P}$ and $\Gamma_{\Delta_P} = \Gamma_{SL}$ by Theorem \ref{thmmain1}, so $\Gamma_{SL} \subseteq \Gamma_{\mathcal{VC}}$. On the other hand $\Gamma_{\mathcal{VC}} \subseteq \Gamma_{SL}$, by Theorem \ref{thmallowedmonodromy}. So $\Gamma_{\mathcal{VC}} = \Gamma_{SL}$. It remains to show that $\mathcal{VC} \subseteq \Delta_P$. Let $\alpha \in \mathcal{VC}$. Then as $T_\alpha$ is the monodromy around a meridian, arguing along the same lines as in the proofs of Lemma \ref{lemgenvan} and Theorem \ref{thmmain1} gives $\alpha = g c_1$ for some $g \in \Gamma_{SL}$. But $g c_1 \in \Delta_P$, so we have shown that $\mathcal{VC} \subseteq \Delta_P$ and hence $\Gamma_{\mathcal{VC}} \cdot \mathcal{VC} \subseteq \Gamma_{SL} \cdot \Delta_P = \Delta_P$. However, $\Delta_P$ is an orbit of $\Gamma_{SL} = \Gamma_{\mathcal{VC}}$, so we must have $\Gamma_{\mathcal{VC}} \cdot \mathcal{VC} = \Delta_P$.
\end{proof}

\section{Application to the topology of Higgs bundle moduli spaces}\label{secappl}

Let $\mathcal{M}(n,d,L)$ denote the moduli space of rank $n$ $L$-twisted Higgs bundles with trace-free Higgs field and determinant equal to a fixed line bundle $D$ of degree $d$. Up to isomorphism, $\mathcal{M}(n,d,L)$ depends on $D$ only through the degree $d$. We can define the Hitchin fibration $h : \mathcal{M}(n,d,L) \to \mathcal{A}$ and one finds that for any value of $d$, the moduli space $\mathcal{M}(n,d,L)$ is a torsor for the family of Prym varieties $p : Prym(\mathcal{S}/\mathcal{A}^{\rm reg}) \to \mathcal{A}^{\rm reg}$ as defined in Section \ref{sec2}. In particular, this implies that the monodromy local system of $\mathcal{M}(n,d,L)$ is $\underline{\Lambda}$, independent of $d$. Let $a \in \mathcal{A}^{\rm reg}$ and let $F_a = h^{-1}(a)$ be the corresponding non-singular fibre of the Hitchin fibration. Our goal this section will be to prove the following:
\begin{theorem}\label{thmrestriction}
Let $\omega \in H^2( F_a , \mathbb{Q} )$ be the cohomology class of the polarization on $F_a$. The image
\begin{equation*}
Im( H^*( \mathcal{M}(n,d,L) , \mathbb{Q} )) \to H^*( F_a , \mathbb{Q} ) )
\end{equation*}
of the restriction map in cohomology is the subspace spanned by $1 , \omega , \omega^2 , \dots , \omega^u$, where $u = dim_{\mathbb{C}}(F_a)$ is the dimension of the fibre.
\end{theorem}

To prove this theorem we need to use a result concerning symplectic vector spaces over finite fields of prime order. Let $p$ be an odd prime, $V$ a vector space over $\mathbb{Z}_p$ of dimension $2v$ and $\langle \; , \; \rangle$ a non-degenerate alternating bilinear form over $V$. Given a symplectic basis $\{ e_1 , f_1, e_2 , f_2 , \dots , e_v , f_v \}$ and $m \in \{0,1,2, \dots , v\}$, let $\alpha_{2m} \in \wedge^{2m} V$ be given by
\begin{equation}\label{equalpha}
\alpha_{2m} = \sum_{ i_1 < i_2 < \dots < i_m } (e_{i_1} \wedge f_{i_1} ) \wedge \dots \wedge (e_{i_m} \wedge f_{i_m} ). 
\end{equation}
It can be shown that $\alpha_{2m}$ is independent of the choice of symplectic basis and is invariant under the group $Sp( V , \mathbb{Z}_p)$ of symplectic transformations of $V$ \cite{de}.
\begin{lemma}\label{leminvs}
The subspace of $\wedge^* V$ invariant under $Sp(V , \mathbb{Z}_p)$ is spanned by $1 , \alpha_2 , \alpha_4 , \dots , \alpha_{2v}$.
\end{lemma}
\begin{proof}
We use induction on the dimension $2v$ of $V$. Assume the result holds in dimension $2v-2$. Choose a symplectic basis $B = \{e_1,f_1 , \dots , e_v , f_v\}$ for $V$. For $i = 1,2, \dots , v$, let $V_i$ be the subspace spanned by $B \setminus \{ e_i , f_i \}$, let $\iota_i : V_i \to V$ be the inclusion and $\pi_i : V \to V_i$ the projection with kernel spanned by $e_i,f_i$.\\

Let $\lambda \in \wedge^k V$ be invariant. Consider first the case where $k$ is odd. For any $i$ we have that $\pi_i(\lambda)$ is invariant under $Sp( V_i , \mathbb{Z}_p)$, so $\pi_i(\lambda) = 0$ by induction. It follows that $\lambda = e_i \wedge \alpha + f_i \wedge \beta + e_i \wedge f_i \wedge \gamma$, where $\alpha,\beta,\gamma$ are in the image of $\iota_i : \wedge^* V_i \to \wedge^* V$. But $\lambda$ is invariant under the transvections $T_{e_i},T_{f_i}$ and it follows easily that $\alpha = \beta = 0$, so that $\lambda$ is a multiple of $e_i \wedge f_i$. Since $i$ was arbitrary, we have that $\lambda$ is a multiple of $e_1\wedge f_1 \wedge \dots \wedge e_v \wedge f_v$. But $\lambda$ has odd degree so this can only happen if $\lambda = 0$.\\

Consider the case where $k = 2m$ is even. Then for each $i$, we have that $\pi_i( \lambda)$ is invariant under $Sp(V_i , \mathbb{Z}_p)$, so by induction we have $\pi_i(\lambda) = c_i \pi_i(\alpha_{2m})$ for some $c_i \in \mathbb{Z}_p$. Clearly this can only happen if $c_1 = c_2 = \dots = c_m$ and thus $\pi_i( \lambda -c_1 \alpha_{2m}) = 0$ for all $i$. By the same argument as used in the case where $k$ is odd, we see that either $\lambda - c_1 \alpha_{2m} = 0$, or $m = v$ and $\lambda-c_1 \alpha_{2v}$ is a multiple of $e_1 \wedge f_1 \wedge  \dots \wedge e_v \wedge f_v = \alpha_{2v}$. In either case $\lambda$ is a multiple of $\alpha_{2m}$.
\end{proof}

\begin{lemma}\label{leminvcohom}
The space $H^*( F_a , \mathbb{Q} )^{\rho_{SL}}$ of monodromy invariant rational cohomology classes on $F_a$ is spanned by $1,\omega , \omega^2 , \dots , \omega^u$.
\end{lemma}
\begin{proof}
Suppose that $\mu \in H^k( F_a , \mathbb{Q} )^{\rho_{SL}}$ is a monodromy invariant cohomology class on $F_a$. Multiplying $\mu$ by a sufficiently large positive integer, it suffices to assume that $\mu \in H^k( F_a , \mathbb{Z} )^{\rho_{SL}}$. Let $p$ be an odd prime not dividing $n$ and let $V = \Lambda_P^* \otimes_{\mathbb{Z}} \mathbb{Z}_p$. Then since $H^k( F_a , \mathbb{Z}) = \wedge^k \Lambda_P^*$, we have $H^k( F_a ,\mathbb{Z}_p) = \wedge^k V$. By Corollary \ref{corpolarization}, the polarization type of $\langle \; , \; \rangle$ on $\Lambda_P$ is $(1,1 , \dots , 1 , n , n , \dots , n)$, where there are $g$ copies of $n$. It follows that the mod $p$ reduction of $\langle \; , \; \rangle$ is a non-degenerate alternating bilinear form on $\Lambda_P \otimes_{\mathbb{Z}} \mathbb{Z}_p$ and so by duality induces a non-degenerate alternating bilinear form $\langle \; , \; \rangle$ on $V$ preserved by the monodromy action. Reduction mod $p$ thus induces a homomorphism $\phi : \Gamma_{SL} \to Sp( V , \mathbb{Z}_p)$ and the reduction of $\mu$ mod $p$ gives an element $\mu_p \in \wedge^k V$ invariant under $\phi(\Gamma_{SL})$. Since $p$ is odd, it follows from \cite[Theorem 2.7 and Theorem 6.5]{bh1} and the fact that $\Gamma_{SL}$ is generated by transvections in the set $S_P$ given in Section \ref{secvanlat}, that $\phi$ is actually surjective. Then by Lemma \ref{leminvs}, we have that $\mu_p = 0$ if $k$ is odd and $\mu_p$ is a multiple of $\alpha_{2m}$ if $k = 2m$ is even.\\

Suppose first that $k$ is odd. Then $\mu$ is divisible by infinitely many primes, hence $\mu = 0$. Now suppose that $k = 2m$ is even. Then for every odd prime $p$ not dividing $n$ we have that the mod $p$ reduction of $\mu$ is a multiple of $\alpha_{2m}^{(p)}$ (here we use a superscript $p$ to remind us that $\alpha_{2m} = \alpha_{2m}^{(p)}$ depends on $p$). Suppose also that $p > m$. Then from Equation (\ref{equalpha}), we have that $\alpha_{2m}^{(p)}$ is the mod $p$ reduction of $\omega^m/m!$, where $\omega \in \wedge^2 \Lambda_P^*$ is the alternating form $\langle \; , \; \rangle$ thought of as a $2$-form. Let $\tau \in \wedge^{2m} \Lambda_P^*$ be a primitive vector such that $\omega^m/m!$ is an integral multiple of $\tau$. The mod $p$ reduction of $\mu$ is a multiple of the mod $p$ reduction of $\tau$ for infinitely many primes $p$. It follows that $\mu$ is in the $\mathbb{Z}$-span of $\tau$ and thus $\mu$ is a rational multiple of $\omega^m$ in $\wedge^{2m} \Lambda_P^* \otimes_{\mathbb{Z}} \mathbb{Q}$.
\end{proof}

\begin{proof}[Proof of Theorem \ref{thmrestriction}]
As the restriction map $r : H^*(\mathcal{M}(n,d,L) , \mathbb{Q} ) \to H^*( F_a , \mathbb{Q} )$ factors through $\mathcal{M}^{\rm reg}(n,d,L)$, it is clear that the image of $r$ is contained in $H^*( F_a , \mathbb{Q} )^{\rho_{SL}}$, the subgroup of monodromy invariants. By Lemma \ref{leminvcohom}, $H^*( F_a , \mathbb{Q} )^{\rho_{SL}}$ is spanned by $1,\omega , \dots , \omega^u$. On the other hand since $\mathcal{M}(n,d,L)$ is quasi-projective, there exists a class $\alpha \in H^2( \mathcal{M}(n,d,L) , \mathbb{Q})$ whose restriction $r(\alpha)$ to $F_a$ is a K\"ahler class, hence non-zero. Thus $r(\alpha)$ must be some non-zero rational multiple of $\omega$. It follows that the image of $r$ is precisely the span of $1,\omega , \omega^2 , \dots , \omega^u$.
\end{proof}

\begin{remark}
When $n$ and $d$ are coprime and $L=K$, Theorem \ref{thmrestriction} may also be proved as follows. By Theorem 7 of \cite{mar}, the cohomology ring of $\mathcal{M}(n,d,K)$ is generated by the K\"unneth factors of the Chern classes of the universal $PGL(n,\mathbb{C})$-Higgs bundle on $\mathcal{M}(n,d,K) \times \Sigma$. Then by a generalisation of the proof of Proposition 5.1.2 in \cite{chm}, one can determine the image of the generators on restriction to a non-singular fibre. The advantage of our proof is that it applies without restriction on the values of $n$ and $d$ and does not require knowing a set of generators for the cohomology of $\mathcal{M}(n,d,L)$.
\end{remark}

\bibliographystyle{amsplain}

\end{document}